\documentclass[12pt]{amsart}

\usepackage{graphicx}

\usepackage[active]{srcltx}

\newcommand{\R}{{\mathbb R}}

\newcommand{\C}{{\mathbb C}}
\newcommand{\Q}{{\mathbb Q}}
\newcommand{\Z}{{\mathbb Z}}

\newcommand{\PP}{{\mathbb P}}

\newcommand{\CP}{\C\PP}

\newcommand{\dbar}{\bar\partial}

\renewcommand{\phi}{\varphi}

\newcommand{\acal}{\mathcal{A}}

\newcommand{\dcal}{\mathcal{D}}

\newcommand{\hcal}{\mathcal{H}}

\newcommand{\jcal}{\mathcal{J}}

\newcommand{\lcal}{\mathcal{L}}

\newcommand{\ocal}{\mathcal{O}}

\newcommand{\rcal}{\mathcal{R}}
\newcommand{\scal}{\mathcal{S}}

\newcommand{\zcal}{\mathcal{Z}}

\newcommand{\half}{\frac{1}{2}}

\newtheorem{mainprop}{{\sc Proposition}}

\newtheorem{theo}{{\sc Theorem}}[section]
\newtheorem{maintheo}{{\sc Theorem}}

\newtheorem{lem}[theo]{{\sc Lemma}}

\newtheorem{prop}[theo]{{\sc Proposition}}

\newtheorem{rem}[theo]{{\sc Remark}}

\title[Gaussian Beams on Zoll manifolds and maximally degenerate Laplacians ]{Gaussian Beams on Zoll manifolds
and maximally degenerate Laplacians}

\author{Steve Zelditch}
\address{Department of Mathematics, Northwestern University,
Evanston IL,  60208-2730, USA}

\thanks{Research partially supported by NSF grant DMS-1206527}

\date{\today}

\begin{document}

\begin{abstract}  Gaussian beams exist along all closed geodesics of  a Zoll surface,
despite the fact that the algorithm for constructing them assumes that the closed
geodesics are non-degenerate.  Similarly, there exists a global Birkhoff normal 
for a Zoll Laplacian despite the degeneracy.  We explain why both  algorithms work
in the Zoll case and give an exact formula for the sub-principal normal form invariant. 
In the case of ``maximally degenerate" Zoll Laplacians, this invariant vanishes and we
obtain  new geometric constraints on such Zoll metrics.
\end{abstract}

\maketitle

\section{Introduction}

A {\it Gaussian beam} is a  sequence $\{\phi_{k}^{\gamma}\}_{k =1}^{\infty}$  of exact or approximate  eigenfunctions of the Laplacian $\Delta_g$ of a compact Riemannian manifold $(M, g)$  which  concentrates  along a stable closed geodesic $\gamma$.The
simplest Gaussian beams  have the approximate form $e^{i k s} e^{- k |y|^2}$ where $(s, y)$ are Fermi normal coordinates with
respect to $\gamma$, with $y $ the normal coordinates and $s$ the arc-length coordinate along $\gamma$. Thus
Gaussian beams  oscillate along $\gamma$ but have Gaussian decay in the orthogonal
direction.  More generally for each $q = 0, 1, 2 \dots, $ one may define higher  Gaussian beams
$\{\phi_{k, q}^{\gamma}\}$ which are of the form $e^{i k s} U_q(k y)$ where $U_q$ is
the qth Hermite function.  Gaussian beams have a long history in mathematics and physics, and we refer to J. Ralston's articles \cite{Ra,Ra2,Ra3} as well as \cite{BB,W} for the some of the  results. 

In all of the articles known to the author, the closed geodesic is assumed to be non-degenerate
 stable elliptic. That is, the eigenvalues of its linear Poincar\'e map $P_{\gamma}$ (defined below) are
assumed to be of unit modulus $e^{i \alpha_j}$ and the Floquet exponents $\{\alpha_j, 2\pi \}$ are assumed to be
independent over $\Q$, in particular to satisfy  $e^{i (m - n) \alpha_j} \not= 1$. The non-degeneracy condition is  needed to prove existence of solutions of certain
transport equations. It arises   for the same reason in the homological equation that appears when one tries
to put $\Delta_g$ into quantum Birkhoff normal form around $\gamma$ \cite{G2,Z2}. However, non-degeneracy
is not always a necessary condition for existence of Gaussian beams or normal forms: in  the case of $S^2$,
$P_{\gamma}$ is the identity operator and its closed geodesics are degenerate. Nevertheless it possesses Gaussian
beams around all closed geodesics. They are the highest  weight spherical harmonics $Y^{\ell}_{\ell}$ of degree $\ell$, which concentrate along the equator.  By rotating  $Y^{\ell}_{\ell}$ one  obtains a Gaussian beam  concentrating
along any closed geodesic of $S^2$.  We refer to \cite{SHang,Z4} for background and references.

Moreover,  there exist Gaussian beams along any closed geodesic of any Zoll surface, i.e. a surface all of whose geodesics are closed \cite{Be}, where again $P_{\gamma} = I$ for all $\gamma$. Unlike the case of the standard metric on $S^2$, the Gaussian beams are usually not eigenfunctions
of $\Delta_g$ but only {\it quasi-modes} or approximate eigenfunctions. However, they are
actual eigenfunctions in the case of  maximally degenerate Laplacians (see \S \ref{MDLSUB}). 
One of our goals is to find the obstructions to constructing Gaussian beams which solve the
eigenvalue equation to arbitrarily high order,
 and analyzing how they can vanish in the maximally degenerate case. 

As discussed in \cite{Z1,Z2}, the construction of Gaussian beams around a closed
geodesic $\gamma$ is closely related to   the
construction of  a quantum Birkhoff normal form for $\Delta$ in a tubular neighborhood around $\gamma$. There are obstructions to the construction of this normal form and apriori they might
not exist in the Zoll case. But
 in \cite{G},  Guillemin constructs  a {\it global} Birkhoff normal form for the
Laplacian on a Zoll surface. Hence the microlocal normal forms also exist (and coincide with
the restrictions of the global one). Parallel to Gaussian beams, one of the aims of this note is to indicate why the constructions are possible 
for Zoll Laplacians  even though they fail completely to satisfy the non-degeneracy assumptions. 

The existence of the global quantum Birkhoff normal form is due to the fact that Zoll geodesic flows are
symplectically equivalent. The symplectic equivalence is due to the fact that one may define a symplectic 
space of geodesics $G(S^2, g)$ of a Zoll surface $(S^2, g)$. To define $G(S^2,g)$ we note that the geodesic flow of a Zoll surface $(S^2, g)$  defines a free $S^1$ action on the unit cotangent
bundle $S^*_g S^2$ for the metric $g$ \cite{GrGr}. The space of geodesics is defined as the orbit space,
\begin{equation} \label{GSg} G(S^2, g): = S^*_g S^2 / S^1. \end{equation}
It is naturally a symplectic manifold \cite{W,Be}. Indeed, tangent vectors to geodesics are orthogonal Jacobi fields,
and the symplectic form is defined by the Wronskian (see \eqref{WR} of \S \ref{JF}).  By the Moser method, one
may construct a symplectic diffeomorphism from $G(S^2, g)$ to the standard space $G(S^2, g_0)$ of geodesics. 
It lifts to a contact transformation $\chi: S^*_g S^2 \to S^*_{g_0} S^2$ of the unit cosphere bundles intertwining
the geodesic flows. The symplectic map $\chi$ may be quantized to give a unitary Fourier integral operator $U = U_{\chi}$ which conjugates
the Laplacians modulo a remainder of order zero, i.e.
$U \Delta_g U^* = \Delta_0 + A_0, $  where $A \in \Psi^0(S^2)$. Here and henceforth, $\Psi^m(M)$ denotes
the class of mth order pseudo-differential operators on a manifold $M$. Starting with any quantization one can improve it by an infinite sequence of pseudo-differential quantizations
so that $A_g$ commutes with $\Delta_g$. The images $\{U \phi_{k, q}^{\gamma}\}$ of the Gaussian beams
for $\Delta_0$ are then Gaussian beam eigenfunctions for $\rcal^2= \Delta_g - A_0$. 

The  global quantum Birkhoff normal form is the statement that,
\begin{equation} \label{A} \sqrt{\Delta_g} = \rcal + A_{-1}, \;\; \mbox{or equivalently}\;\;\Delta_g = \rcal^2 + A_0  \end{equation}
where $\rcal$ has the eigenvalues $k + \half$ with multiplicity $2k + 1$ (just as for the
standard sphere), and 
where $A_{-1}$  (resp. $A_0$) is a pseudo-differential operator of order $-1$ (resp. zero)  commuting with $\rcal$ \cite{G,W}. 
This commutation implies that the  symbol $\sigma_{A_{-1}}$ of $A_{-1}$   (resp. $\sigma_{A_0}$ of $A_0$) is invariant under
the geodesic flow. We may then interpret it as a function $H$ on the space $G(S^2, g)$ of geodesics of the Zoll surface.
The eigenvalues of $A_{-1}$ in the kth cluster (where $\rcal = k + \half)$ have the form
$\mu_k(q)$ where $q = - k, \dots, k$. Thus the spectrum of $\sqrt{\Delta_g}$ is a small perturbation
$k + \half  + \mu_k(q)$ of that of the standard metric.  The eigenvalue distribution of $A_{-1}$ is determined by
the function $H$. Explicit formulae are given in \cite{Z1,Z3}. In effect we give yet another formula in this article.


\subsection{\label{MDLSUB} Maximally degenerate Laplacians}

A  Laplacian of a surface  is called   {\it maximally degenerate } (in the sense of \cite{Z1})
if  the multiplicities of the distinct eigenvalues of the  Laplacians 
are identical to the multiplicities $(1, 3,5, ..., 2k + 1...)$ of the eigenvalues $k(k + 1)$ of the
Laplacian   for the standard metric on $S^2$.  It is proved in  \cite{Z1} that in dimension 2, the metrices of  MDL (maximally degenerate Laplacian) metrics must be
Zoll metrics on $S^2$. Further it is proved that $A_{-1}$ and $A_0$ in \eqref{A} are smoothing operators. The main question is whether any non-standard Zoll Laplacians exist \cite{Y}.   It is known
that the standard metric is the only MDL among surfaces of revolution \cite{E,Z1}.  But most Zoll metrics
are not $S^1$ invariant. There exists an infinite dimensional moduli space of Zoll metrics, whose tangent space at the standard
metric is  isomoprhic to the space of  odd functions on $S^2$.

The question whether there exist non-standard maximally degenerate Zoll Laplacians is of interest (to the author)
because it is a very simple inverse spectral problem to state, yet  defies most of the known techniques, and 
in some sense is a test for the strength of these techniques.  The same question may be posed in higher dimensions,
and then it is not even known if the standard metric on $S^n$ is determined by its eigenvalues (with multiplicities) if $n \geq 7$. 
It has been proved that the standard metric on $S^n$ is locally determined by its spectrum and therefore does not
admit isospectral deformations. But it is not even known if the multiplicites of the eigenvalues of the standard $S^2$
are rigid, i.e. whether there exists a non-trivial Zoll deformation of the standard metric which preserves all of the
eigenvalue multiplicities.

The relevance of Gaussian beams to this inverse problem is that, in the MDL case, Gaussian beams are  eigenfunctions of the Laplacian. But   the Gaussian beam construction involves an ansatz which leads to an eikonal equation and various
transport equations. There exist obstructions to solving the transport equations, and they  must vanish
in the MDL case since Gaussian beams exist. The main purpose of this note is to explain how to write down
such obstructions explicitly in terms of geometric invariants.  The vanishing of the obstructions
gives  new conditions that an MDL must satisfy. 

In fact, it was proved in \cite{Z1} that the eigenvalues 
of an  MDL Zoll surface agree  with the standard eigenvalues $\ell(\ell + 1)$  up to rapidly decaying errors.  So the obstructions
can be determined exactly.  This idea is parallel to the strategy in \cite{Z2}, where the spectral projections $\Pi_{\ell}(x, y)$ kernels
associated to the eigenspace were studied. If $y$ is fixed, then $x \to \Pi_{\ell}(x, y)$ is the analogue of a zonal
(rotationally invariant) eigenfunction on $S^2$ and may be viewed as a quasi-mode associated to the Lagrangian
submanifold formed by meridian geodesics through the pole $y$.  The obstructions to the Gaussian beam construction
turn out to be quite  different from the obstructions found in \cite{Z3}  to the construction of  $ \Pi_{\ell}(x, y)$. 

 Our main result, Theorem \ref{GBO},  states that if $\Delta_g$ is an MDL then the the integral
over every closed geodesic $\gamma$ of   certain polynomials
in the curvature and Jacobi fields must vanish.   It turns out that the the polynomials are identical to the
ones that arise in the  algorithm for constructing  a quantum Birkhoff normal form for $\Delta$ around $\gamma$. 
In fact the normal forms algorithm and the algorithm for constructing Gaussian beams are equivalent and the
obstructions are the same. They are simpler to compute in the normal forms algorithm and so after reviewing
the Gaussian beam construction we mainly concentrate on normal forms.


\subsection{Gaussian beams and Birkhoff normal form} To state the results, we first recall 
 the construction of excited Gaussian beams
along  stable  elliptic geodesics. We follow the exposition of \cite{BB}, which constructs Gaussian beams for every
eigenvalue $q$  of a transvese Harmonic oscillator.   The $q$th excited Gaussian beam along an elliptic closed geodesic $\gamma$ has the
asymptotic  form
\begin{equation} \label{GB} \Phi_{kq}(s,\sqrt{r_{kq}}y) =e^{ir_{kq}s} \sum_{j=0}^{\infty}
r_{kq}^{-\frac{j}{2}} U_q^{\frac{j}{2}}(s,
\sqrt{r_{kq}}y,r_{kq}^{-1}).\end{equation} In unscaled coordinates we denote them by $\phi_{kq}^{\gamma}$. As  recalled in \S \ref{JF}, 
In the non-degenerate case, the  semi-classical parameter $r_{kq}$ has the form,
\begin{equation} \label{rkq} r_{kq} =  k +  \half + \frac{1}{2 \pi} \sum_{j=1}^n (q_j + \frac{1}{2}) \alpha_j\end{equation}
where as above, $e^{i \alpha_j}$ are the eigenvalues of the Poincar\'e map $P_{\gamma}$. The coefficient functions
 $U^{\frac{j}{2}}$ are obtained by solving transport equations. The construction of $\Phi_{kq}(s,\sqrt{r_{kq}}y)$
produces a  sequence of approximate  eigenvalues associated to $\gamma$ whose square roots have  the semi-classical expansion
\begin{equation} \label{EIGint} \lambda_{kq} \equiv r_{kq} + \frac{p_1(q)}{r_{kq}} + \frac{p_2(q)}{r_{kq}^2} + ...\end{equation}
The numerators $p_n(q) $ are polynomials of specific degrees in $q$. We will express $p_n(q)
$ as the eigenvalues of functions $p_n(I)$ where $I$ are transvese Harmonic oscillators
to $\gamma$, and will also consider  the Weyl symbols $p_n(|z|^2)$ of $p_n(I)$. We will  will show that $p_1(|z|^2)$ has the form $c_4 |z|^4 + c_0$. It is not quite obvious that they are spectral
invariants of $\Delta_g$ but in fact they are equivalent to the wave trace invariants along $\gamma$ (\cite{G2,Z2}).

Roughly speaking, one can write the transport equations in the form,
\begin{equation} \label{TER}   \frac{\partial}{\partial s}  (\mu(a) U^{\frac{j}{2}}_q) = RHS, \end{equation}
where $\frac{\partial}{\partial s}$ is differentiation along $\gamma$ and $\mu$ is a certain unitary transform. 
The obstruction to solving these equations is simply that  $\int_{\gamma} \left( \mbox{RHS} \right) ds = 0$.
As we found in \cite{Z1,Z2} (and will review below), the  RHS is a
complicated polynomial in the curvature $\tau$, its normal derivatives $\tau_{\nu}$ along $\gamma$ and
in the normal  Jacobi fields $Y = y \nu$ along $\gamma$. Here $\nu$ is the unit normal frame along $\gamma$.

In the Zoll case,  on constructs the normal form for $\Delta_g - A_g$ and all of the subprincipal terms vanish. One has $\alpha_j = 0$  and so $$r_{kq} = k + \half $$
for all $k,q$.
Similarly, one can construct Gaussian
beams along any closed geodesic  satisfying 
$$(\Delta_g - A_0 + (k + \half)^2) \phi^{\gamma}_k = O(k^{-\infty}), $$
which implies that 
$$ (\Delta_g + (k + \half)^2) \phi_{k}^{\gamma} = O(1), \;\; (k \to \infty). $$
There is already an obstruction to constructing such quasi-modes of order $O(1)$  of the form
 \begin{equation} \label{CUBIC} \int_{\gamma} y^3 \tau_{\nu} dt \equiv 0, \;\; (\forall \gamma),  \end{equation}
where  $y \nu$
is a normal Jacobi field along $\gamma$. 
It is the simplest example of an obstruction  \eqref{TER}. 
But the idenity holds for any  geodesic of any Zoll surface as a result of Jacobi's equation (see 
  (4.21) of   \cite{Z1}). This fact is an early indication of why the construction of Gaussian beams is possible
on Zoll surfaces. 

The principal symbol $H$ of $A_0$  can be obtained as the limit,
$$H(\gamma) = \lim_{k \to \infty}  \langle A_{0} \phi_k^{\gamma}, \phi_k^{\gamma} \rangle. $$
The limit only involves the ground state $q = 0$ Gaussian beams. One may obtain `higher
symbols' by studying the limit along pairs $(k, q)$ with $\frac{q}{k} \to \alpha$. 
We refer to \cite{UZ} for some results in this direction. Viewed as a function on the space of geodesics $\simeq G(S^2, g)$, the eigenvalues
$\mu_k(q)$ are roughly the values of $H$ on the qth Bohr-Sommerfeld level with Planck
constant $\frac{1}{k}$. The proof is based on putting $H$ into Birkhoff normal form (i.e. 
express it in terms of local action variables) and is in some sense related to the quantum
normal form of this article.  

In the MDL case one may construct Gaussian beams which are exact eigenfunctions with eigenvalue
$k (k + 1) + S(k)$ where $S(k)$ is rapidly decaying.  This leads to many further vanishing obstructions of
the type \eqref{CUBIC}. 
One  must separate out  integral formulae that hold universally on Zoll surfaces
from ones which put extra conditions on $g$ in the MDL case.

\subsection{Constraints on maximally degenerate Zoll metrics}

It is shown in \cite{Z1,Z3} that, in the maximally degenerate Zoll case, all   $p_j(q)$ in \eqref{EIGint}  are
independent of $q$ and depend only on $k$.  In dimension 2, all of the terms below $r_{kq} $ in the expansion
\eqref{EIGint} are zero. More precisely, we have

\begin{mainprop} \label{EXIST} \cite{Z1,Z3}

For a maximally degenerate Zoll surface, $A_{-1}$ in \eqref{A} is a smoothing operator, and there exists for each $q = 0, 1, 2 \cdots$ a sequence of $\Delta$-eigenfunctions
with the asymptotic form \eqref{GB}. The associated eigenvalues have the expansions \eqref{EIGint} with
$r_{kq} =  
(k +\half)$ and with all $p_n(q) = 0$. Thus, for every closed geodesic,
$\lambda_{kq} = k (k + 1) + S(k)$ where $S(k)$ is rapidly decaying. 
 
\end{mainprop}

This puts an infinite number of additional constraints on a Zoll metric in the form of vanishing integrals over 
every closed geodesic of certain metric invariants. The simplest non-universal one is given in

\begin{maintheo}  \label{GBO} Let $g$ be a Zoll  metric on $S^2$. Then a necessary 
and sufficient condtion for existence of
   Gaussian beams \eqref{GB} satisfying 
$$ (\Delta_g  - A_0 + (k + \half)^2) \Phi_{kq}^{\gamma} = O(k^{-2}), \;\; (k \to \infty) $$
 along $\gamma$  is that 
 $p_1(q) = 0$ for all closed geodesics. At  the closed geodsic $\gamma(t)$
with $t \in [0, \pi]$,  the coefficients $c_j$ $(j = 0, 2)$ of the Weyl symbol $p_1(|z|^2) = c_2 |z|^4 + c_0$  of $p_1(I)$ are given by the Weyl symbol $\sigma_{A_0}$ plus  
\begin{equation} \label{YINTS} \begin{array}{l} \frac{1}{2\pi}\int_0^{2 \pi}  [a_j |\dot{Y}|^4 +   b_{1j} \tau |\dot{Y}Y|^2 +
b_{2j}  \tau  \Re (\bar{Y} \dot{Y})^2
+c_{j}  \tau^2 |Y|^4 + d_j \tau_{\nu \nu} |Y|^4 +e_j \delta_{j0} \tau ] ds \\ \\
+\frac{1}{2\pi}\sum_{0\leq m,n \leq 3; m+n=3} C_{2;mn j} \Im \{\int_0^{2 \pi} \tau_{\nu}(s)\bar
{Y}^mY^n(s)[\int_0^s\tau_{\nu}(t)\bar {Y}^nY^m](t)dt] ds\} .  \end{array} \end{equation}
Here,  $Y$ is the complex periodic normal Jacobi field on $\gamma$ with $Y(0) = 1, Y'(0) = i$, and $a_j, b_{1j}, b_{2j}, c_j, d_j, e_j, C_{2; mnj}$ are certain universal non-zero constants which  depend on
the index $j$. One has $e_j = 0$ if $j = 2$ and $d_j = 0$ if $j = 0$. 

In the MDL case
the integrals \eqref{YINTS} vanish for all $\gamma$.

\end{maintheo}
In fact, there are many further integrals of the same kind which must vanish in the MDL case. 

We emphasize that the normal form exists for 
$\Delta_g - A_g = \rcal(\rcal + 1)$. They  can be used to determine the complete symbol of $A_g$. However,
the emphasis of this note is on the MDL case and we do not consider general Zoll surfaces in detail.  

Below we give a detailed algorithm for determining the terms and the universal coefficients. However it is rather
detailed and messy and it useful to have indirect checks on the calculations. 
 In the case of the round metric,
$\tau \equiv 1$ and  the  condition simplifies  to 
$$\begin{array}{l} \frac{1}{2\pi}\int_0^{2 \pi}  [a |\dot{Y}|^4 +   b_1  |\dot{Y}Y|^2 +
b_2   Re (\bar{Y} \dot{Y})^2
+c  |Y|^4 + e \delta_{j0}  ] ds   = 0.  \end{array}$$
Setting  $y = e^{i t}$ gives a  universal  linear relation among the universal constants
in the formula for $j = 0, 2$. 

 For MDL's,  $p_n(q) = 0$ for all $n$ and all $\gamma$ and that  gives an  infinite sequence of
identities  involving ever more derivatives which must vanish along every closed geodesic. 
Unfortunately the formulae for the higher polynomials $p_n(q)$ of \eqref{EIGint}  rapidly become extremely complicated. In
the inverse spectral theory of bounded plane domains one can sift out spectral  invariants by iterating,
i.e. by replacing $\gamma$ by its kth iterate $\gamma^k$. But in the Zoll case the integrand is periodic and
iterating does not seem to produce new invariants.

As will be discussed in \S \ref{J} (see \S \ref{RELATIONS}), there are relations among the  terms in the above integral due
to integral geometric identities on Zoll surfaces.  The result above should be compared with other
formulae in \cite{Z1,Z3}. In Theorem 3 of \cite{Z1} the principal symbol $\sigma_{A_{-1}} = H$ is
evaluated at $\gamma$ to equal a universal multiple of
\begin{equation} \label{H} H(\gamma) = \int_{\gamma} \left( \tau + \left[\frac{1}{3} \tau_{\nu}(s) y^3(s) \int_0^s \tau_{\nu}(t) J^3(t) dt -
\tau_{\nu}(s) u^2(s)  J(s) \int_0^s \tau_{\nu} u J^2 dt \right] \right) ds,  \end{equation}
which is similar to the expression of Theorem \ref{GBO} but is simpler in having 
no fourth degree terms in the Jacobi fields. Here, $u$ is the Jacobi fields with $u(0) = 1,
\dot{u}(0) = 0$ and $J$ is the area density in normal coordinates, which is essentially
the Jacobi field with $J(0) = 0, \dot{J}(0) = 1$.  The integral is obtained by regularizing
$\int_{\gamma} J^{\half} \Delta J^{-\half} ds. $

In our view, the universal
coefficients cannot be calculated reliably without a computer, and therefore we do not try to use the
exact expression in Theorem \ref{GBO}  to draw conclusions in inverse spectral theory. Instead we rely
on indirect arguments and on multiple calculations of the same quantity using different
techniques. For instance the calculation of  \eqref{H} in \cite{Z3}  is quite independent of
Theorem \ref{GBO} and suggests that there exist cancellations in the terms for $j = 0$. 
As reviewed in \S \ref{J} the calculation in \cite{Z3} is based on the construction of
`zonal quasi-modes' or spectral projections kernels in the MDL case. Unlike the Gaussian
beam construction, where one constructs $\Phi_{k,q}^{\gamma}$ for all $\gamma, q$,
we only considered the exact projection kernel and not a family of quasi-modes depending
on the same Lagrangian submanifold, which might require a normal form along the Lagrangian.

\subsection{Normal form invariants}

The expression in Theorem  \ref{GBO}  is the quantum Birkhoff normal form of degree 2 (i.e. $p_1(q)$ in \eqref{EIGint})
and was calculated in  \cite{Z2}  for a non-degenerate elliptic closed geodesic.  
In  the non-degenerate case there are  some additional terms which vanish in the Zoll case. 
 In \cite{Z2}  we proved:\bigskip

\noindent{\bf QBNF coefficients for k=0, dim =2}{\;\; \it  $p_1(q)$ is given in complex Fermi normal 
coordinates $z = y + i\eta$ as   $B_{0;4}|z|^4 + B_{0;0}$ where $B_{0; j}$
are given for both $j=0,4$ by polynomials in the normalized orthogonal Jacobi eigenvectors $Y$ of $P_{\gamma}$ by
$$B_{0;j} =   \frac{1}{L}\int_0^L [a |\dot{Y}|^4 +   b_1 \tau |\dot{Y}Y|^2 +
b_2 \tau  Re (\bar{Y} \dot{Y})^2
+c \tau^2 |Y|^4 + d \tau_{\nu \nu} |Y|^4 +e \delta_{j0} \tau ] ds + $$
$$ +\frac{1}{L}\sum_{0\leq m,n \leq 3; m+n=3} C_{1;mn}\frac{\sin((n-m)\alpha)}
{|(1 - e^{i(m-n)\alpha})|^2}
|\int_0^L  \tau_{\nu}(s)\bar{Y}^mY^n](s)ds|^2 $$
$$+\frac{1}{L}\sum_{0\leq m,n \leq 3; m+n=3} C_{2;mn}Im \{\int_0^L \tau_{\nu}(s)\bar
{Y}^mY^n(s)[\int_0^s\tau_{\nu}(t)\bar {Y}^nY^m](t)dt] ds \}.$$ }
\medskip
As mentioned above, the coefficients $C_{1; mn}$ and $C_{2;mn}$ are universal and come from commutator
identities in the 
Weyl calculus (see \S \ref{C}). In dimension 2, and in the non-degenerate case, the space of  complex normal Jacobi
fields is spanned by the normalized eigenvectors $\{Y, \bar{Y}\}$. 

The coefficient of  the middle term of the normal form in the non-degenerate elliptic case would blow up
in the Zoll case. Nevertheless the normal form exists in every Zoll case. The resolution is that the  corresponding coefficient automatically vanishes  for every closed geodesic. In the case of the second
normal form invariant, the bad coefficient is \eqref{CUBIC} and vanishes   for every Zoll surface.

\bigskip

\noindent{\bf Acknowledgements}  We thank the referee for comments and questions that helped improved
the exposition.

\section{\label{JF} Jacobi fields and Poincar\'e map on a Zoll manifold}

In this section we review the properties of Jacobi fields on Zoll manifolds.
Let $\gamma$ be a closed geodesic of length $L_{\gamma}$  of a Riemannian manifold $(M, g)$. 
  We denote by  ${\jcal}_{\gamma}^{\bot}
\otimes \C$  the space of complex normal Jacobi fields along $\gamma$, a symplectic
vector space of (complex) dimension 2n (n=dim M-1) with respect to the Wronskian
\begin{equation} \label{WR} \omega(X,Y) = g(X, \frac{D}{ds}Y) - g(\frac{D}{ds}X, Y). \end{equation}
 The linear Poincare map $P_{\gamma}$ is  the  symplectic map on ${\jcal}_{\gamma}
^{\bot} \otimes \C$ defined by $$P_{\gamma} Y(t) = Y(t + L_{\gamma}).$$  

The closed geodesic is {\it elliptic} if the eigenvalues of $P_{\gamma}$ are
of the form $\{ e^{\pm i \alpha_j}, j=1,...,n\}$. The associated normalized eigenvectors
will be denoted $\{ Y_j, \overline{Y_j}, j=1,...,n \}$,
\begin{equation} \label{PEV} P{\gamma} Y_j = e^{i \alpha_j}Y_j \;\;\;\;\;\;P_{\gamma}\overline{Y}_j =
e^{-i\alpha_j} \overline{Y}_j \;\;\;\; \omega(Y_j, \overline{Y}_k) = \delta_{jk} \end{equation}
 and relative to a fixed parallel
normal frame $e(s):= (e_1(s),...,e_n(s))$ along $\gamma$ they will be written in the form
$Y_j(s)= \sum_{k=1}^n y_{jk}(s)e_k(s).$  As mentioned in
the introduction, it is usually assumed in normal forms or Gaussian beams constructions that
$\gamma$ is non-degenerate elliptic, i.e.  
 $\{   \alpha_j,  j=1,...,n\}$ together with $\pi$,  are independent over ${\bf Q}$.

 In this article, we mainly consider surfaces, in which
case   ${\jcal}_{\gamma}^{\bot} \otimes \C$ has complex dimension two and as mentioned above
is spanned by the eigenvectors $\{Y, \bar{Y}\}$.  A normal Jacobi field along
$\gamma$  is simply of the   form $Y(s)= y(s) \nu (s),$ 
where  $\nu(s)$ is the parallel unit normal vector $\gamma$.  Jacobi's equation
is then a second order scalar equation,
$$  y'' + \tau  y = 0.$$
There is a two dimensional space of solutions: the vertical Jacobi
field $y_1$ with initial conditions $y(0) = 0, y'(0) = 1$ and the
horizontal Jacobi field $y_2$ with initial conditions $y(0) = 1,
y'(0) = 0$ with respect to a fixed choice of origin $\gamma(0)$ of $\gamma$. We consider the pair $(y, y')$ and form the symplectic   Wronskian matrix:
\begin{equation} \label{as} a_s:= \left( \begin{array}{ll}  y_2'(s)  \;\;\;& y_1'(s)\\  y_2(s)
\;\;\;&  y_1(s) \end{array} \right). \end{equation}
We modify the   Wronskian matrix so that its columns are given in terms of the
 normalized eigenvectors \eqref{PEV}  of the Poincar\'e map:

\begin{equation} \label{ACAL} \acal(s):=\left( \begin{array}{ll}  \Im\dot{Y} & \Re \dot{Y} \\  \Im Y &
\Re Y
\end{array}\right). \end{equation} The somewhat strange positioning of the elements is to maintain
consistency with our reference  \cite{Fo} on the metaplectic representation.



In the Zoll case, $P_{\gamma} = Id$, i.e. the normal Jacobi fields are periodic  (all $\alpha_j = 0$) and
the Wronskian matrices $a_s$ resp. $\acal(s)$  are periodic. The $a_s$ matrix is uniquely determined
but $\acal(s)$ is not since all normal Jacobi fields are eigenvectors of eigenvalue $1$. We may thus
assume 
\begin{equation} \label{Yy} Y  = y_1 + i y_2,\end{equation}
so that $a_s = \acal(s)$.

Since we have a family of closed geodesics, we may differentiate with respect
to the family. If we deform the geodesic in the direction of its unit normal, we obtain the variation $y_{\nu}$,
which satisfies (cf. \cite{Z1}, p. 573) 
\begin{equation} \label{JNU} y_{\nu}'' + \tau_{\nu} y^2 + \tau y_{\nu} = 0. \end{equation}To see this, we  let $\gamma_r(t)$ be the variation 
of $\gamma$ defined by $y \nu$ and then consider the family of Jacobi fields $Y_r = y_r \nu_r$ along $\gamma_r$
with $Y_0 = y\nu$. Since $y_r'' + \tau(\gamma_r(t)) y_r = 0$ we obtain the formula by differentiating
with respect to $r$ at $r = 0$ using that $\frac{d}{dr}|_{r = 0}  \gamma_r(t) = y(t). $ The notation
$y_{\nu}$ is short for $\frac{d}{dr} y_r |_{r = 0}. $

The following Lemma  illustrates universal integral formulae on Zoll surfaces. 

\begin{lem} \label{CUBE} On a Zoll surface, for any normal Jacobi fields $y \nu, y_2\nu $ one has  $\int_{\gamma} \tau_{\nu} y^2 y_2  ds = 0$ for all closed geodesics $\gamma$. \end{lem}
\begin{proof}

If we multiply \eqref{JNU} by $y_2$, integrate over $\gamma$ and transfer the $\frac{d}{ds}$ derivatives from $y_{\nu}''$
to $y_2$, then by Jacobi's equation $\int_{\gamma} (y''_2 y_{\nu} + \tau y_2  y_{\nu})  ds = 0$ and therefore
\eqref{CUBIC}  is a universal identity for Jacobi fields along
geodesics of Zoll surfaces. 
\end{proof}

\subsection{Fermi normal coordinates}

 Fermi normal coordinates are the normal coordinates defined by the exponential map $\exp: N_{\gamma, \epsilon} \to T_{\epsilon}(\gamma)$ from a ball
in the  normal
bundle of $\gamma$ to a tube of radius $\epsilon $ around $\gamma$. Thus we write $(s, y) = \exp_{\gamma(s)} y \cdot \nu_{\gamma(s)}, $ where $\nu(s)$ is a choice of unit normal frame along $\gamma$. We write the associated metric coefficients as $g_{00} = g(\partial_s, \partial_s), g_{0j} = 0$
and $g_{jk} = g(\partial_{y_j}, \partial_{y_k}) = 1$.

In Fermi coordinates along a geodesic, the  field $\frac{\partial}{\partial s}$ is a horizontal Jacobi field
pointing between nearby normal geodesics to $\gamma$ and tangent to the wave fronts. 
Each $\frac{\partial}{\partial y_j} $ is a geodesic vector field.  The volume density is given by  $j = \sqrt{\det g} $.  In dimension two,  $j = ||\frac{\partial}{\partial s}||.$

\section{Gaussiam beam quasi-modes}

The main result of \cite{BB, Ra2,Ra3} is the construction of quasi-modes of the
form \eqref{GB} along a non-degenerate stable elliptic closed geodesic. We write
\begin{equation} \label{GB2} U_{kq}(s, \sqrt{r_{kq}}y,r_{kq}^{-1}) = \sum_{j=0}^{\infty}
r_{kq}^{-\frac{j}{2}} U_q^{\frac{j}{2}}(s,
\sqrt{r_{kq}}y,r_{kq}^{-1}),\end{equation} 
and seek  $ U_q^{\frac{j}{2}}$ so that \eqref{GB} approximately solves the eigenvalue problem. 
\begin{equation} \label{EIGPROB} \Delta_y e^{ir_{kq}s} U_{kq}(s, \sqrt{r_{kq}}y,r_{kq}^{-1}) \sim
\lambda_{kq}e^{ir_{kq}s}
 U_{kq}(s, \sqrt{r_{kq}}y,r_{kq}^{-1})\end{equation}
The eigenvalues associated to $\gamma$  have the semi-classical expansion \eqref{EIGint}  
where $r_{kq}$ is given by \eqref{rkq}.
We now review the construction, following \cite{BB}.

It is sometimes convenient to express the various functions in terms of the scaled coordinates $\mu = \sqrt{r_{kq}} y$. The ground state Gaussian beam is a locally defined function in Fermi normal coordinates $(s, \mu)$ along
$\gamma$ of the form,
$$U_0 = \frac{1}{\sqrt{\det Y}}e^{\frac{i}{2} \left( \langle P Y^{-1}  \mu,  \mu \rangle \right)}. $$
Here, $Y$ is the matrix of Jacobi eigen-fields and $P$ is the matrix whose columns are $\frac{dY_j}{ds}$. 
Thus, $P = Y', PY^{-1} = Y' Y^{-1}$ and
$$Y^*P - P^* Y = i I \;\;\; Y^T P - P^T Y = 0. $$

Using transverse creation and annihilation operators one can construct  higher excited
states as  Gaussian beams  $U_q$ where
the  subscript $q = 0, 1, 2, \dots$  denotes the `transverse
energy level', i.e. the energy level of a Hermite operator in the normal directions to $\gamma$. 
When $q = 0$, the  $j = 0$ term  is defined by  \begin{equation} \label{U0} U_q^0 = U_0(s,\mu)  = (\det
Y(s))^{-1/2} e^{i\half \langle \Gamma(s) \mu, \mu\rangle}
\end{equation}
 where $\Gamma(s) := \frac{dY}{ds} Y^{-1} $ (see
\cite{Ra,BB}).  For higher $q$, the initial term $U^0_q$ is like the
$q$th excited state of the transverse harmonic oscillator and is defined by
$$U_q^0 = \Lambda^q U_0, $$ where 
$$\begin{array}{ll}
\Lambda =\sum_{k=1}^2 (i y_{k} D_{y_k} -\frac{dy_{k}}{ds} y_k) &
\Lambda^*=\sum_{k=1}^{2} (-i \overline{y}_{k}D_{y_k} -
\overline{\frac{dy_{k}}{ds}}y_k) \end{array} $$ are the transverse creation/annihilation
operators  adapted to $\gamma$. Write
$$U_r = \Lambda_1^{* r_1} \cdots \Lambda_m^{* r_m} U_0 = Q_r( \mu, s) U_0. $$
Here, $r =(r_1, \dots, r_m)$.  When $M $ is a  surface, there is just one $k$ index. 

As we now indicate, $U_{kq}$ and  $ U_q^{\frac{j}{2}}$ are  obtained by solving a sequence of transport equations. 

\subsection{Semi-classical scaling}



To determine the coefficients $U^{\frac{j}{2}}_q$ in \eqref{GB}, one  rescales the
Laplacian to convert \eqref{EIGPROB} into a semi-classical expansion.
It is convenient to replace  $\Delta$ by the unitarily equivalent  1/2-density Laplacian
 $$\Delta_{1/2} := j^{1/2} \Delta j^{-1/2},$$
which can be written  in the form:
$$\Delta_{1/2} = j^{-1/2}\partial_s g^{00}J \partial_s j^{-1/2}
+\sum_{ij =1}^{n} j^{-1/2}\partial_{y_i} g^{ij} J \partial_{ y_j}
j^{-1/2}$$
$$\equiv g^{00}\partial_s^2 + \Gamma^0 \partial_s +
 \sum_{ij=1}^n g^{ij} \partial_{u_i}\partial_{y_j} + \sum_{i=1}^{n} \Gamma^{i}
\partial_{y_i} + \sigma_0.$$
 From now on, we denote $\Delta_{\half}$  simply by $\Delta.$

In dimension 2, 

$$-\Delta = J^{-1/2}\partial_s g^{00}J \partial_s J^{-1/2}
+ J^{-1/2}\partial_{y}  J \partial_{ y}
J^{-1/2}$$
$$\equiv g^{00}\partial_s^2 + (\partial_s g^{ss}) \partial_s +
  \partial_{y}^2+  \sigma_0 = g^{00}\partial_s^2 + (\Gamma^s_{ss}) \partial_s +
  \partial_{y}^2+  \sigma_0 $$

For simplicity of notation, we put
$$h_{kq}:= (2\pi k + \sum_{j=1}^n (q_j +
\frac{1}{2}\alpha_j))^{-1}. $$
We observe that in \eqref{GB} the Gaussian beams are in scaled coordinates $\mu = \sqrt{h_{kq}}^{-\half} y,$
$$\Phi_{kq}(s, h_{kq}^{-\frac{1}{2}} y) = e^{i\frac{ s}{h_{kq} L}} U_{kq}(s,
\sqrt{h_{kq}}^{-\half} y ,h_{kq}) $$ and the eigenvalue
problem \eqref{EIGPROB} becomes
$$\Delta_u e^{\frac{i}{h_{kq}L}s}
U_{kq} (s, h_{kq}^{-\half} y ,h_{kq}) = \lambda(h_{kq})
e^{\frac{i}{h_{kq}L}s}U_{kq} (s, h_{kq}^{-\half} y ,h_{kq}).
$$ When indices are not needed we simply write $h = h_{kq}$.


We  transfer the scaling from the unknown function $U$ to the Laplacian using
the    unitary operators
$T_h$ and $M_h$ 
$$T_h (f(s,y)|ds|^{1/2}|dy|^{1/2}):= h^{-n/2} f(s, h^{-\half}y) |ds|^{1/2}|dy|^{1/2}$$
$$M_h(f(s,y)|ds|^{1/2}|dy|^{1/2}) := e^{\frac{i}{hL}s} f(s,y)|ds|^{1/2}|dy|^{1/2}.$$
We easily see that:
\begin{equation} \label{SCALING}  \left\{ \begin{array}{l} 
  T_h^* D_{y_j} T_h = h^{-\half} D_{ y_j} \\ \\
 T_h^* y_i T_h = h^{\half} y_i, \\ \\
 M_h^*D_s M_h =((hL)^{-1} + D_s).  \end{array} \right. \end{equation}

We then rescale an operator $A$ by
\begin{equation} \label{Ah} A_h := T_h^*M_h^*AT_hM_h. \end{equation}
The rescaled Laplacian then has the form,
$$\begin{array}{lll} -\Delta_{h} & = &  -(h)^{-2} g^{00}_{[h]} + 2i(h)^{-1}g^{00}_{[h]}\partial_s + i(h)^{-1}\Gamma^0_{[h]} \\&&\\&&+
 h^{-1}( \sum_{ij=1}^n g^{ij}_{[h]}\partial_{y_i}\partial_{y_j}) + h^{-\half}(\sum_{i=1}^{n} \Gamma^{i}_{[h]}
\partial_{y_i}) + (\sigma)_{[h]}, \end{array}$$
 the subscript $[h]$ indicating to dilate the coefficients of the operator in the form,
$f_h(s, y):=f(s, h^{\half} y).$
Expanding the coefficients in Taylor series at $h=0$, we obtain
the asymptotic expansion
\begin{equation} \label{DELTAH} \Delta_h \sim \sum_{m=0}^{\infty} h^{(-2 +m/2)}{\lcal}_{2-m/2} \end{equation}
where ${\lcal}_2=  1,$ ${\lcal}_{3/2}=0$ and where
\begin{equation} \label{lcal} \lcal_1 =: \lcal=  2 [i  \frac{\partial}{\partial s} + \half \{\sum_{j=1}^{n}
\partial_{y_j}^2 -
\sum_{ij=1}^{n} K_{ij}(s) y_i y_j\}]. \end{equation}
Here, $K_{ij}$ are the sectional curvatures. When the dimension is two,
$$\lcal=  D_s - \frac{1}{2} ( D_{y}^2 +\tau(s)
y^2).$$

The semi-classical eigenvalue problem $ \Delta_h U(s,  y,h) = \lambda(h)  U(s,
y,h) $ then becomes,

\begin{equation} \label{EIGPROB3} \begin{array}{l}
 \left(\sum_{m=0}^{\infty} h^{(-2 +m/2)}{\lcal}_{2-m/2} \right)  U(s,  y,h) = \lambda(h)  U(s,y,h) )  \end{array}\end{equation}  with
\begin{equation} U(s, y, h) =  \sum_{j=0}^{\infty}
h^{\frac{j}{2}} U_q^{\frac{j}{2}}(s,
y , h).\end{equation}   

By construction, 
$$\lcal U_q^0 = (q + \half) U_q^0. $$
The remaining terms are determined by `transport equations'  \eqref{TRANSPORT} of the form,
\begin{equation} \label{TE} \lcal U^{\frac{j}{2}}_q = \mbox{RHS}, \end{equation}
for a known RHS.

The equation for $U^{h_0}_q$ is then shown to take the form of a sequence of transport
equations with respect to the `parabolic operator' $\lcal$, 
\begin{equation} \label{TRANSPORT} \lcal U_q^{h_0} = - \lcal_1 U_q^{h_0-1} - \cdots - \lcal_{h_0} U^0_q =: \Psi^{h_0}_q U_0,  \end{equation}
where $\Psi_q^{h_0}$ is a polynomial in $y$ with smooth coefficients in $s$. 
We next recall how this works.

\subsection{Gaussian beam obstructions}
In this section, we  outline the algorithm in \cite{BB} for determining the  numerators $p_n(q)$ in \eqref{EIGint}. They 
are chosen to make transport equations solvable. We do not give the algorithm in detail because the
equivalent  Birkhoff normal forms algorithm
is more efficient, in that it works simultaneously for all $q$.

We express the right side of \eqref{TRANSPORT} in the form
\begin{equation} \label{PSIA} \Psi_q^{h_0}(y, s) U_0 = \sum_{(r)} A_{q r}^{h_0} U_r(y, s), \end{equation}
where (by orthogonality),
$$A_{q r}^{h_0} = \frac{1}{r! (2 \pi)^{m/2}} \int_{\R^m} \Psi_q^{h_0} U_0 \overline{U}_r d  \mu. $$

It follows that
$$A_{qr}^{h_0}(s + L) = e^{i (\kappa_q - \kappa_r) } A_{qr}^{h_0}(s), $$
with
$$\kappa_q = - \sum_{j = 1}^m (\half  + q_j)\alpha_j, \;\;\; \kappa_r = - \sum_{j = 1}^m (\half  + r_j)\alpha_j. $$
In the Zoll case,  $A_{qr}^{h_0}$ is periodic. 
If  $h_0$ is odd, we find that 
$A_{qq}^{h_0} (s) = 0$.




Write
\begin{equation} \label{UqB} U_q^{h_0} (y,s)= \sum_{(r)} B_{qr}^{h_0}(s) U_r = \Phi_q^{h_0}(y, s) U_0.  \end{equation}
In order that $U_q^{h_0} e^{i \kappa_q s}$ be periodic it is sufficient that
$$B_{qr}^{h_0} (s + L) = B_{qr}^{h_0} (s) e^{i (\kappa_q - \kappa_r)}. $$
One has  $\lcal U_r = 0$ and the transport equations \eqref{TRANSPORT} simplify to
\begin{equation}\label{BA} 2 i \frac{d}{ds} B_{qr}^{h_0} (s)= A_{qr}^{h_0} (s).  \end{equation}



When  $B_{qq}^{h_0}$ is periodic  the necessary and sufficient condition
for solvability is
$$\int_0^L A_{qq}^{h_0} ds = 0. $$
The numerators $p_n(q) $ of \eqref{EIGint}  are chosen to make this equation hold. 

Thus, to determine the explicit geometric obstructions we would need to calculate $A_{qq}^{h_0}$ explicitly,
and this is quite messy. We therefore turn to the normal forms construction, which calculates the same
obstructions in a somewhat simpler way.

\section{Quantum Birkhoff normal form construction}

In this section we review the quantum normal form construction from \cite{Z2}. As mentioned above, 
in the standard algorithm it is assumed
there that the closed geodesic $\gamma$ is non-degenerate elliptic.  In this article, we want
to understand why the algorithm still works in the 
Zoll case and how  it must be modified.  The modifications are emphasized in remarks. 

Let us summarize the key points of the normal form construction.   The goal is to 
 conjugate  $\sqrt{\Delta}$ microlocally around $\gamma$   to a function
of the tangential operator $D_s:=\frac{\partial}{i
\partial s}$  along $\gamma$  together with the transverse harmonic
oscillators
$$I_j=I_j(y,D_y) := \frac{1}{2} (D_{y_j}^2 + y_j^2) .$$
Here, $(s, y)$ are Fermi normal coordinates along $\gamma$. We assume that the length $L$ of $\gamma$ equals
$2 \pi$.
Given the Floquet exponents, we form the operator
$$\rcal = D_s + H_{\alpha}$$
where
$$H_{\alpha}:= \frac{1}{2}\sum_{k=1}^n \alpha_k I_k. $$

\begin{rem}
In the Zoll case, $H_{\alpha} = 0$  is zero and hence the Zoll Laplacian is conjugated to a function of $D_s$ alone. 
\end{rem}

In the non-degenerate case one has (\cite{Z2}):

\begin{theo}  \label{QBNF}  There exists a microlocally elliptic
  Fourier Integral operator $W$
from a conic neighborhood of $\R^+\gamma$ in $T^*N_{\gamma}-0$ to a conic neighborhood
of $T^*_{+}S^1$ in $T^*(S^1\times R^n)$ such that:

$$W \sqrt{\Delta_{\psi}}W^{-1} \equiv 
[ {\mathcal R} +\frac{ p_1(I_1,....,I_n)}{{\mathcal R}}
+\frac{p_2(I_1,...,I_n)}{({\mathcal R})^2} + ...+\frac{p_{k+1}(I_1,\dots,I_n)}{({\mathcal R})^{k+1}}+\dots]$$
$$ \equiv D_s + H_{\alpha} + \frac{\tilde{p}_1(I_1,\dots,I_n)}{ D_s} +
 \frac{\tilde{p}_2(I_1, \dots,I_n)}{(D_s)^2}
+\dots+\frac{\tilde{p}_{k+1}(I_1,\dots,I_n)}{( D_s)^{k+1}}+ \dots$$
where the numerators $p_j(I_1,...,I_n), \tilde{p}_j(I_1,...,I_n)$ are polynomials of degree
j+1 in the variables $I_1,...,I_n$. \end{theo}
\medskip

As explained in detail in \cite{Z2}, the $\equiv$ sign refers to a doubly-graded class   $O_m \Psi^r$ of operators
(or symbols). Here,   Also, $O_m \Psi^r$
denotes the space of pseudodifferential operators of order $r$ whose complete symbols
vanish to order $m$ at $(y,\eta)=(0,0)$. The
kth remainder term of the right sides of Theorem \ref{QBNF}  lie in the space $\bigoplus_{j=0}^{k+2} O_{2(k+2-j)}\Psi^{1-j}$.

The polynomials are the same as the numerators of \eqref{EIGint}.  We now go over the normal form algorithm
in the non-degenerate case. We then consider how to modify it so that it produces a normal form for a Zoll Laplacian
and in particular for am maximally degenerate one.

\subsection{Metaplectic Jacobi conjugation}

We work inductively on the semi-classical expansion \eqref{DELTAH}. 
The first step is to put $\lcal$ \eqref{lcal} into normal form along $\gamma$. We do
this by using a ``moving metaplectic conjugation'.

Let  $\mu$ denote the metaplectic representation (see \cite{Fo} for background).
  We apply $\mu$ to the matrix \eqref{ACAL} of Jacobi fields to obtain,
$$\mu(\acal):=\int_{\gamma}^{\oplus}\mu(\acal(s))ds $$
on $\int_{S^1}^{\oplus} L^2(\R^n) ds = \int_{\gamma}^{\oplus} L^2(N_{\gamma(s)})ds$. In other words,
$$\mu(\acal) f(s,y) = \mu(\acal(s))f(s,y)$$
where the operator on the right side acts in the $y$-variables. This conjugation simplifies the quadratic
term \eqref{lcal} of the semi-classical expansion \eqref{DELTAH}.

\begin{prop} \label{MUL} The image $\lcal$ of $D_s$ under $\mu$ is the operator \eqref{lcal}:
  $$\lcal:= \mu(\acal)^{*} D_s \mu(\acal) = D_s - \frac{1}{2}(\sum_{j=1}^n D_{y_j}^2 + \sum_{ij=1}^n K_{ij}(s)
y_i y_j).$$ 

\end{prop}
\medskip

The relation of this conjugation to the Gaussian beam construction is as follows:

(i) $$\mu(\acal^{-1}) \gamma_0(s,y):=U_0(s,y)  = (\det Y(s))^{-1/2}
e^{i\half \langle \Gamma(s) y,y\rangle}$$ where $\Gamma(s) :=
\frac{dY}{ds} Y^{-1}.$

(ii) If $\gamma_q$ denotes the $q$ Hermite function, then  $$\mu(\acal^{-1})\gamma_q:=U_q =
\Lambda_1^{q_1}...\Lambda_n^{q_n} U_0. $$

Lemma \ref{MUL} shows that the first non-trivial term of  \eqref{DELTAH} simplifies  if we conjugate   by $\mu(a)$.
We therefore conjugate the entire expansion.   We denote the new operators by
\begin{equation} \label{DCALH} \begin{array}{lll} \dcal_h & = & \mu(\acal)^{-1} \Delta_h \mu(\acal) \\&&\\
& \sim&  \sum_{m=o}^{\infty} h^{(-2 + \frac{m}{2})}
\dcal_{2 -\frac{m}{2}}  \end{array} \end{equation}
  with  $\dcal_2 = I, \dcal_{\frac{3}{2}}=0,
\dcal_1=D_s$.

This  conjugation will make the transport equations
take the form $D_s \dcal_p = \mbox{RHS}$ and thus easy to solve. On the other hand, for explicit calculations
we need to use the $\lcal_q$'s.

\subsection{Microlocal conjugation}
We now microlocally conjugate \eqref{DCALH} to the Birkhoff normal form of
Theorem \ref{QBNF}.  
The intertwining opertor in \cite{Z2} has  the form,
\begin{equation} \label{Wh} W_h:=\mu(\acal)^*\prod_{k=1}^{\infty}W_{h \frac{k}{2}}\mu(a) \end{equation} with
$$W_{h \frac{k}{2}}:= \exp(ih^{\frac{k}{2}} Q_{\frac{k}{2}}). $$
Here, $Q_{\frac{k}{2}}$ is a polynomial differential operator which is constructed simultaneously with
the normal form. 

 As discussed in \cite{Z2}, the $Q_{\frac{j}{2}}$ are constructed so that the conjugation
removes all terms in the expansion \eqref{DCALH} except for functions of the actions $I_j$. The procedure is algebraic,
and is simplified by using the Weyl calculus because of its equivariance under metaplectic conjugation.

Each time we conjugate by one more factor of \eqref{Wh}, we change the terms of the expansion \eqref{DCALH},
and we need to introduce some notation for the new expansions. We only carry out the algorithm to two orders. 
The first conjugation  (by $\exp(ih^{\frac{1}{2}} Q_{\frac{1}{2}}) $)  produces
$$\begin{array}{lll} \dcal^{\half}_h: & = &\tilde{W}_{h \frac{1}{2}}^* \dcal_h \tilde{W}_{h \frac{1}{2}} 
 \sim  \sum_{n=0}^{\infty} h^{-2 + \frac{n}{2}} \sum_{j+m=n} \frac{i^j}{j!}
(ad\tilde{Q}_{\half})^j \dcal_{2 - \frac{m}{2}} \\&&\\&&
= h^{-2}L^{-2} + h^{-1}L^{-1}D_s +\sum_{n=3}^{\infty}
 h^{-2 + \frac{n}{2}} \dcal^{\half}_{2 - \frac{n}{2}}. \end{array} $$

We then conjugate once more with $ \tilde{W}_{h 1}= e^{i h \tilde{Q}_1},$
 to define $$\begin{array}{lll} \dcal^1_h: & = & \tilde{W}_{h 1}^*\dcal^{\half} \tilde{W}_{h 1} = h^{-2}L^{-2} + h^{-1}L^{-1}D_s +
h^{-\half}\dcal^{\half}_{\half} + \dcal^1_0(s,D_s, x,D_x) + \dots, \end{array}$$

Thus, the superscript of $\dcal^a_b$ indexes  the last factor of \eqref{Wh} used in the conjugation and
the subscript is the opposite of the power of  $h$,  $h^{-2 + \frac{n}{2}}$. 

\subsection{\label{C} Commutators and Weyl symbols}

In the Weyl calculus, commutators are  given symbolically by the odd expansion
$$a\#b - b\#a \sim \frac{1}{i} P_1(a,b) + \frac{1}{i^3 3!}P_3(a,b) +\dots$$
while anticommutators involve only the even transvectants.   One easily computes that
$$P_1(z^m\bar{z}^n, z^{\mu}\bar{z}^{\nu}) = C_{1;mn\mu\nu} z^{m+\mu -1}\bar{z}^{n + \nu -1}$$
where $C_{1;mn\mu\nu} = \half\sigma((m,n),(\mu,\nu))$ with $\sigma$ the standard symplectic inner product, and
that
$$P_3(z^m\bar{z}^n, z^{\mu}\bar{z}^{\nu}) = C_{(m,n),(\mu,\nu)} z^{m+\mu - 3}\bar{z}^{n + \nu -3}$$
for certain other coefficients $C_{(m,n),(\mu,\nu)}$. These are the $C$-coefficients in Theorem \ref{GBO}.

\subsection{First odd term} To begin with,
we construct $Q_{\half}$ so that conjugation by $\exp(ih^{\frac{1}{2}} Q_{\frac{1}{2}}) $  removes $\dcal_{\half}$ in \eqref{DCALH}, i.e. so that $\dcal_{\half}^{\half} = 0$.  This happens if $Q_{\half}$ solves
the {\it homological equation}
$$\{ i [D_s,\mu(r_{\alpha})^*Q_{\frac{1}{2}}\mu(r_{\alpha})]+
{\dcal}_{\frac{1}{2}} \}|_0= 0,$$
or equivalently,
\begin{equation} \label{FIRSTT} \partial_s\{\mu(r_{\alpha})^*Q_{\frac{1}{2}}\mu(r_{\alpha})\}|_0 =
-i  \{{\dcal}_{\frac{1}{2}}\} |_0. \end{equation}
Here,  
\begin{equation} \label{r} r_{\alpha_j}(s):= \left( \begin{array}{ll} \cos \alpha_j \frac{ s}{L} & \sin \alpha_j
\frac{ s}{L} \\ -\sin \alpha_j \frac{ s}{L} & \cos \alpha_j \frac{ s}{L}
 \end{array} \right). \end{equation}
where  $L$ is the length of $\gamma$. Since our main application is to Zoll surfaces, where the closed
geodesics all have the same primitive period \cite{GrGr}, we usually set $L  = 2 \pi$ or $1$ to simplify notation. In
the Zoll case they are the same. Also,  ``$A|_0$" denotes
restriction  of a pseudo-differential  operator $A $  on $S^1 \times \R^n$  to elements
in the kernel of $\rcal$.   Equivalently,
 after conjugation by $\mu(r_{\alpha})$, 
 to elements in the kernel of $D_s$ , that is, to functions
independent of $s$. Thus, if 
$A = A_2 D_s^2 + A_1 D_s + A_0$, then $A|_0 = A_0|_0$.

To solve \eqref{FIRSTT},  we rewrite it in terms of  complete Weyl symbols.
We denote by  $A(s,x,\xi)$  the complete Weyl symbol of the
operator $A(s,x,D_x)$.  Then \eqref{FIRSTT}  becomes
\begin{equation} \label{FIRSTTb} \partial_s \tilde{Q}_{\half}(s,x,\xi)= -i {\dcal}_{\half}|_0(s,x,\xi)
\end{equation}
with
$$\tilde{Q}_{\half}(s + L,x,\xi) = \tilde{Q}_{\half}(s,r_{\alpha}(L)(x,\xi)).$$
Here $L $ is the length of $\gamma$ (which we usually take to be $2 \pi$). 
We rewrite \eqref{FIRSTTb}  in the integral form
\begin{equation} \label{eq1} \tilde{Q}_{\half}(s,x,\xi) = \tilde{Q}_{\half}(0,x,\xi) + L \int_0^s
-i {\dcal}_{\half}|_0(u,x,\xi)du \end{equation}
and then  $\tilde{Q}_{\half}(0,x,\xi)$ is determined by the consistency condition
$$\tilde{Q}_{\half}(L,x,\xi) - \tilde{Q}_{\half}(0,x,\xi) =
L \int_0^L -i {\dcal}_{\half}|_0(u,x,\xi)du. $$
or in view of the periodicity properties of $Q$, by 
\begin{equation} \label{eq1b} \tilde{Q}_{\half}(0,r_{\alpha}(x,\xi)) - \tilde{Q}_{\half}(0,x,\xi) =
L \int_0^L -i {\dcal}_{\half}|_0(u,x,\xi) du. \end{equation}

\begin{rem} In the Zoll case, both sides of this equation must be zero and thus a solution of \eqref{eq1} 
is given by
$$\tilde{Q}_{\half}(s,x,\xi) =  L \int_0^s
-i {\dcal}_{\half}|_0(u,x,\xi)du.  $$
 \end{rem}

To solve \eqref{eq1b}, we  change   to complex
coordinates $z_j = x_j + i\xi_j$ and $\bar z_j = x_j - i \xi_j$ in which the
action of  $a(s)$ or $r_{\alpha}(L)$ is diagonal.
 Then \eqref{FIRSTTb} becomes,
$$\tilde{Q}_{\half}(0,e^{i\alpha}z, e^{-i\alpha}\bar z) -
\tilde{Q}_{\half}(0,z,\bar z) =
\int_0^L -i {\dcal}_{\half}|_o(u,z,\bar z ) du$$
We now use that ${\dcal}_{\half}(u,z,\bar z )$ is a polynomial of degree 3, and put
\begin{equation} \label{HALFS} \tilde{Q}_{\half}(s,z,\bar z) = \sum_{|m|+|n|\leq 3} q_{\half;mn}(s) z^m \bar z^n, \;\;{\dcal}_{\half}|_0(s,z,\bar z ) du = \sum_{|m|+|n|\leq 3} d_{\half;mn}(s)
z^m \bar z^n \end{equation}
then \eqref{FIRSTTb}  or \eqref{eq1b} becomes
\begin{equation} \label{FIRSTTpoly} \sum_{|m|+|n|\leq 3} (1 - e^{(m - n) \alpha}) q_{\half; mn}(0) z^m \bar z^n
= -i \sum_{|m|+|n|\leq 3} \bar d_{\half;mn}
z^m \bar z^n . \end{equation}
Here, the bar in $\bar d_{\half;mn}$ denotes the time average over one period. 
The obstruction to solving this equation if the right side is non-zero  is invertibility of the coefficients $(1 - e^{(m - n) \alpha})$.
Since there are no terms with $m=n$ in this (odd-index) equation,  there is no obstruction to
the solution of \eqref{FIRSTTpoly} if
the $\alpha_j$'s are independent of $\pi$ over $\Z$.

\begin{rem} \label{REM1} In the Zoll case, both sides of \eqref{FIRSTTpoly}   equal zero by Lemma \ref{CUBE}  (as mentioned above).
 We  solve with
$q_{\half ;mn}(s) = L \int_0^s d_{\half; mn}(s) ds.  $ It follows that $q_{\half; m,n}(0) = 0. $ \end{rem}

\subsection{First even term}

Next we  seek $\tilde{Q}_1$
so that conjugation by $ \tilde{W}_{h 1}= e^{i h \tilde{Q}_1}$ removes as much
as possible of the first even term of the conjugate $\dcal_h^{\half}$ of \eqref{DCALH} by $\exp(ih^{\frac{1}{2}} Q_{\frac{1}{2}}) $ in
the previous step. This step becomes computationally involved since we need to conjugate  \eqref{DCALH}  
by  $\exp(ih^{\frac{1}{2}} Q_{\frac{1}{2}}) $ and introduce notation for the new terms. For simplicity of exposition
we refer to \cite{Z2} for the details of all the steps and just explain the notation for the first two conjugations. 


 In the second conjugation, 
we seek an element $\tilde{Q}_1(s,x,D_x)\in \Psi^*(S^1 \times \R^n)$ and a function
$f_0(I_1,...,I_n) $ of the Harmonic Oscillators so that
$${\dcal}^1_h:=\tilde{W}_{h 1}^*{\dcal}^{\half} \tilde{W}_{h 1} = h^{-2}L^{-2} + h^{-1}L^{-1}D_s +
h^{-\half}{\dcal}^{\half}_{\half} + {\dcal}^1_0(s,D_s, x,D_x) + \dots$$
with
\begin{equation} \label{DCAL0} {\dcal}_0^1(s,D_s,x, D_x)|_0 = f_0(I_1,...,I_n),  \end{equation}
  Note that due to the choice of the first conjugation,  ${\dcal}^1_{\half} = {\dcal}^{\half}_{\half} = 0.$ We now
choose the second $\tilde{Q}_1$ to remove as much as possible of the $h^0$ term, i.e. all of the ``off-diagional''
terms of the Weyl symbol of $\dcal_0^1$. 
First, we observe that
$$\dcal_0^1 = \dcal_0^{\half} + i [D_s, \tilde{Q}_1].  $$
 Thus, we seek  $\tilde{Q}_1$ so that
$$\{ i [D_s,\tilde{Q}_1] + {\dcal}_0^{\half}\}|_0 = f_0(I_1,...,I_n). $$
 or equivalently
\begin{equation} \label{SECONDTT} \partial_s\tilde{Q}_1|_0 =\{-{\dcal}_0^{\half} + f_0(I_1,...,I_n)\}|_0. \end{equation} 
We further  note that 
\begin{equation} \label{D0HALF} {\dcal}^{\half}_0 = {\dcal}_0 + \frac{i}{2} [\dcal_{\half}, \tilde{Q}_{\half}]. \end{equation}
Recall here that $\dcal_{\half}$ is the term of order $h^{-\half}$ in \eqref{DCALH}
and that $\tilde{Q}_{\half}$ is accompanied
by $h^{\half}$. There is an additional double commuator $ [ [D_s,\tilde{Q}_{\half}], \tilde{Q}_{\half}]$ term,
but since  $[D_s,\tilde{Q}_{\half}] = i {\dcal}_{\half}$ by the first step,  we get 
$$\frac{i}{2} [\dcal_{\half}, \tilde{Q}_{\half}] = i [\dcal_{\half},\tilde{Q}_{\half}] - \frac{1}{2} [ [D_s,\tilde{Q}_{\half}], \tilde{Q}_{\half}].$$

 
 \subsection{\label{HOMOLNOND} Solution of the second homological equation}

We rewrite \eqref{SECONDTT}  in terms of  complete Weyl symbols and obtain the second homological equation,
$$ \partial_s \tilde{Q}_1(s,z, \bar z) = -i \{{\dcal}^{\half}_0|_0 (s,z,\bar z) -
f_0(|z_1|^2,\dots, |z_n|^2)\} $$
or equivalently
$$\tilde{Q}_1(s,z,\bar z) = \tilde{Q}_1(0,z,\bar z) -i \int_0^s
[{\dcal}^{\half}_0|_0 (u,z,\bar z) -
f_0(|z_1|^2,\dots, |z_n|^2)] du$$
and solve simeltaneously for $\tilde{Q}_1$ and $f_0$. The consistency condition
determining a unique solution is that
$$\tilde{Q}_1(L,z,\bar z) = \tilde{Q}_1(0,z,\bar z) -i  \int_0^L
[{\dcal}^{\half}_0|_0 (u,z,\bar z) -
f_0(|z_1|^2,\dots, |z_n|^2)] du. $$
or 
$$\tilde{Q}_1(0,e^{i\alpha}z,e^{-i\alpha}\bar z) - \tilde{Q}_1(0,z,\bar z) = -i  \{\int_0^L
{\dcal}^{\half}_0|_0 (u,z,\bar z)du  -
 f_0(|z_1|^2,\dots, |z_n|^2) \}.$$

Now, ${\dcal}^{\half}_0|_0 (u,z,\bar z)$ is a polynomial of degree
4. Indeed, $\dcal_0(u, z, \bar{z}) $ is a polynomial of degree 4, and  the Weyl symbol
of $[\dcal_{\half}, \tilde{Q}_{\half}]$ is also of degree since it is the Poisson bracket of polynomials
of degree $3$ (see \S \ref{C}). We assume  that the Weyl symbol $ \tilde{Q}_1(s,z,\bar z)$ is a polynomial of degree 4.  We put
$$ \tilde{Q}_1(s,z,\bar z) = \sum_{|m|+|n|\leq 4} q_{1;mn}(s) z^m \bar z^n, \;\;\;\;\;\;
 f_0(|z_1|^2,\dots, |z_n|^2) = \sum_{|k|\leq 2} c_{0 k} |z|^{2k} $$
and
$${\dcal}^{\half}_0|_0 (s,z,\bar z)du :=\sum_{|m|+|n|\leq 4} d_{0; mn}^{\half}(s) z^m \bar
z^n,\;\;\;\;\;\;\;\;\;\bar d^{\half}_{0;mn}:= \frac{1}{L}\int_0^Ld^{\half}_{0;mn}(s)ds. $$

The second homological equation thus becomes, 
\begin{equation} \label{HOMO2} \left\{ \begin{array}{l}
\sum_{|m|+|n|\leq 4, m \not= n} (1 - e^{i (m-n)\alpha}) q_{1;mn}(0) z^m \bar z^n 
=\sum_{|m|+|n|\leq 4, m \not= n} \overline{d_{0; mn}^{\half}} z^m \bar
z^n, \\ \\
\overline{d_{0; mm}^{\half}} =: f_0(I). \end{array} \right.
  \end{equation}

In the non-degenerate case, we can solve for the off-diagonal coefficients,
$$q_{1; mn}(0) = -i  (1 - e^{i (m-n)\alpha})^{-1} \bar d_{0; mn}^{\half}. $$
We cannot divide when $m = n$, and  must set the diagonal coefficients equal to zero.  The coefficients $c_{0 k}$
of $f_0$ (the normal form)
are then determined by
$$c_{0 k} =   \bar d_{0; kk}^{\half}. $$

It is evident that $\tilde{Q}_1$ and $f_0(I_1,\dots,I_n)$ are
even polynomial pseudodifferential  operators of degree 4 in the variables $(x,D_x)$. The coefficients $c_{0 k}$ are 
essentially the  QBNF invariants.

\begin{rem} \label{REM2} In the MDL Zoll case, the left side of the off-diagonal sum $m \not= n$ is zero, and therefore
a necessary condition for solvability is that $ \overline{d_{0; mn}^{\half}}  = 0$ for all $m\not= n$
with $m + n \leq 4$. In the maximally degenerate case, the normal form term $f_0$ is also zero and
therefore we also have $   \bar d_{0; kk}^{\half} = 0$ for $k  = 0, 1, 2$.  

Again we solve for $\tilde{Q}_1$ by direct integration,
$$q_{1; m, n}(s) = \int_0^s (d^{\half}_{0; m, n}(s) - \delta_{m k} \delta_{n k} c_{0k} |z|^{2k} ) ds. $$\end{rem}

.

\section{\label{Z} Maximally degenerate Zoll case}

As we have remarked  in the summary above, and the solvability  of the homological equations
is possible in the Zoll case where   $ (1 - e^{i (m-n)\alpha}) = 0$  only when the right side of the homological equations vanish. In this section we review the results that prove that the normal form does exist and hence the homological
equations are solvable when $\Delta_g$ is maximally degenerate. In the general Zoll case, one needs to subtract
the operator $A_0$ to solve the equations and $A_0$ is determined by the solvability of the equations.

As mentioned in the introduction, it is proved in \cite{W,G} that on  any Zoll surface there exists a global unitary Fourier integral operator $U$ intertwining $\Delta_g$ with the standard Laplacian
modulo a remainder of order zero, i.e.
\begin{equation} \label{UCONJ} U \Delta U^* = \Delta_0 + A_0, \;\;\; U \sqrt{\Delta} U^* = \sqrt{\Delta_0} + A_{-1}. \end{equation}

Thus, $\Delta_g - A_0$ has a complete degenerate  global quantum Birkhoff normal form. It restricts around each closed geodesic
to a microlocal one. Hence the homological equations must be solvable. 


\subsection{\label{MDL} Spectral projections and unitary intertwining operators in the MDL case}

The   unitary intertwining can be constructed directly in terms of Gaussian beams.
First in the standard $S^2$, the spectral projection to the kth eigenspace  $\hcal_k$ of spherical harmonics on $S^2$ of degree $k$ satisfies the following identity,
$$\Pi_k(x, y) = \int_{G(S^2, g_0)} \phi_k^{\gamma}  \otimes \phi_k^{\gamma} d\mu(\gamma). $$
Here, $G(S^2, g_0)$ is the symplectic space of geodesics and $ d\mu(\gamma)$ is the symplectic area form.
Also, $\phi_k^{\gamma}$ is the Gaussian beam (highest weight spherical harmonic) of degree $k$.

One then constructs Gaussian beams on any Zoll surface  by
$U^* \phi_k^{\gamma}$, since
$$\Delta U^* \phi_k^{\gamma} = U (\Delta_0 + Q_0) \phi_k^{\gamma} = \lambda_0(q, h) U \phi_k^{\gamma} + U Q_0 \phi_k^{\gamma}. $$
The second term is $\ocal(1)$.  In general, $U^* \phi_k^{\gamma}$ is only a quasi-mode of order zero.

The unitary intertwining operator $U$ thus takes  $\phi_k^{\gamma}$
to the ground state Gaussian beam $\psi_k^{\chi(\gamma)}$ for $g$  along $\chi(\gamma)$ where $\chi: G(S^2, g_0) \to G(S^2, g)$ is 
a symplectic diffeomorphism. We therefore have,

\begin{prop} Let $\chi: G(S^2, g) \to G(S^2, g_0)$ be a symplectic diffeomorphism \eqref{GSg}. Let $(S^2, g)$ be MDL.
Then, a  unitary intertwining operator for the $k$th cluster is defined by
$$U_k(x, y) =  \int_{G(S^2, g_0)} \psi_k^{\chi(\gamma)}  \otimes \phi_k^{\gamma} d\mu(\gamma). $$
\end{prop}
In the maximally degenerate case, this formula constructs $U$  so that $U \phi_k^{\gamma}$
is an eigenfunction of the Zoll surface.  

\subsection{Maximally degenerate Zoll Laplacians}

We now specialize to the maximally degenerate case. 
It follows from   Proposition \ref{EXIST}  that:

\begin{theo} \label{dq=0} \cite{Z1,Z2}  If $(S^2, g)$ is a Zoll surface with a maximally degenerate Laplacian, then for every closed geodesic,
 all of the coefficents
$p_j(I_1, I_2) $ and $\tilde{p}_j(I_1, I_2)$ equal zero.  
In particular,  $A_{-1} = 0$ and  $ \overline{d_{k; mn}^{\half}}  = 0$ for all $m, n$ with $m + n \leq 4$. Moreover,
$\bar{d}^q_{ m, n} = 0$ for all $q$ and $m, n$.

\end{theo}\

We now derive some explicit geometric consequences from Theorem \ref{dq=0}, in particular Theorem \ref{GBO},
by calculating $d^{\half}_{m, n}$. The diagonal calculations are essentially in \cite{Z2} except that we
need to clarify their solvability in the case of MDL Zoll surfaces. It appears that the vanishing of the off-diagonal
coefficients gives yet futher constraints on $g$ of the same form as in Theorem \ref{GBO} but with different
coefficients. This will become visible in the proof.



\section{Explicit formulae on Zoll surfaces}

We wish to evalute the coefficients $\overline{d}^{\half}_{ m, m}$ or equivalently  the complete symbol of $f_0$  in terms of integrals over
$\gamma$ of Fermi-Jacobi data.  At first, we will allow the dimension to be arbitrary; when it is time
to substitute in  metric expressions we will restrict to dimension 2.

\subsection{Abstract calculation}

To calculate $f_0(|z|^2)$, we need to unravel the diagonal part of the equation \begin{equation}
\label{DCALHALFINT}  \int_0^L
{\dcal}^{\half}_0|_0 (u,z,\bar z)du  .  \end{equation}
The integrand is given in  \eqref{D0HALF} and consists of two terms $ {\dcal}_0 + \frac{i}{2} [\dcal_{\half}, \tilde{Q}_{\half}]$.
We first consider the commutator term.
We note that $\dcal_{\half}$ is independent of $D_s$ so that $\dcal_{\half}|_0 = \dcal_{\half}.$
It follows that
$$ [\dcal_{\half}(s), \tilde{Q}_{\half}(s)] =  [\dcal_{\half}(s), \tilde{Q}_{\half}(0)] +  [\dcal_{\half}(s),\int_0^s \dcal_{\half}(t)dt]$$
so that the second term of \eqref{D0HALF} contributes to $f_0(I_1,\dots,I_n)$ the {\em diagonal part} of
\begin{equation} \label{COMMPART} \frac{i}{2}\{[\tilde{Q}_{\half}(0,e^{i\alpha}z,e^{-i\alpha}\bar{z}), \tilde{Q}_{\half}(0,z,\bar{z})] +
\frac{1}{2}\{[\int_0^L \dcal_{\half}(s)ds, \int_0^s \dcal_{\half}(t)dt]. \end{equation}
Here, the bracket $[,]$ denotes the  commutator of complete symbols in the sense of operator
(or complete symbol) composition. In the Zoll case, $\tilde{Q}_{\half}(0, \cdot, \cdot) = 0$ by Remark
\ref{REM1} and Remark \ref{REM2}. For emphasis:

\begin{rem}\label{COMMZERO} In the Zoll case, $q_{\half; m, n}(0) = 0$ and therefore this commutator term is zero. \end{rem}


\subsection{Explicit calculations of obstruction integrals}

To evaluate  the expressions $\dcal_{\half}$ and $\dcal_0$.
we conjugate back to the $\lcal$'s:
\begin{equation} \label{DCAL0a} \dcal_0 = \mu(\acal^*)\lcal_0 \mu(\acal^*)^{-1},\;\;\;\;\;\;\;\;
\dcal_{\half} = \mu(\acal^*) \lcal_{\half}\mu(\acal^{*})^{-1} \end{equation}
where as in \eqref{ACAL}, 
\begin{equation} \label{ACALb} \acal(s):=\left( \begin{array}{ll}  \Im\dot{Y} & \Re \dot{Y} \\  \Im Y &
\Re Y
\end{array}\right). \end{equation}

\begin{rem} As noted above, $Y$ is uniquely defined in the non-degenerate case but not in the Zoll
case. We return to this point in the next section. \end{rem}

We then conjugate the symbols.  By metaplectic covariance of the Weyl
calculus,
the conjugations change the complete Weyl symbols of the $\lcal$'s (in the $x$ variables) by the linear
symplectic transformation $\acal$, i.e by the substitutions
$$\begin{array}{l} x \rightarrow  [(\Re Y) x + (\Im Y)\xi] = \half [\bar {Y}\cdot z + Y \cdot
\bar{z}]
\\
\xi \rightarrow [\Re \dot{Y} x + (\Im \dot{Y})\xi] = \half  [\bar{\dot{Y}}\cdot z +
\dot{Y}\bar{z}]
\end{array}.
\leqno (7.8)$$

\subsection{Dimension 2}

  In dimension 2 we have (in scaled Fermi coordinates)
$$g^{oo}(s,y) =1+ C_1 \tau(s) y^2 +C_2 \tau_{\nu}(s)y^3 + \dots\;\;\;\;\;\;\;g^{11}=1 \;\;\;\;
J(s,u) = \sqrt{g_{oo}}= 1 + C_1'\tau(s)y^2 + \dots $$
 for some universal (metric independent) constants $C_j,C_j'$ which will change from line to line. They will end up in the
the coefficients of the normal form.

  Using the Taylor expansion of the metric coefficients one finds that 
\begin{lem} \label{LOCCALS} We have,
$$\left\{\begin{array}{l} \lcal_{\half} = C L^{-2}\tau_{\nu}(s) y^3,\;\;\;\;\; \\ \\ \lcal_0 = C_1 L^{-2}y^4 \tau_{\nu \nu}
+ C_2 L^{-1} y^2 \tau \partial_s + C_3 L^{-1} \tau_s y^2 - \partial_s^2+  C_4\tau y \partial_y  + C_5\tau. \end{array} \right.$$ All terms have weight -2.
\end{lem}



We now complete the calculation of the diagonal terms in \eqref{DCALHALFINT}. We rewrite the commutators in terms of the coefficients  \eqref{HALFS}  at $s = 0$ and the Poisson bracket
constants in \S \ref{C}.

As noted in Remark \ref{REM1},Remark \ref{REM2} and Remark \ref{COMMZERO},  $\tilde{Q}_{\half}(0) = 0$ and therefore its commutators
make no contribution in the Zoll case. Thus we only need to calculate the diagonal part of
\begin{equation} \label{DCALCOMM}  \frac{1}{2}\{[\int_0^L \dcal_{\half}(s)ds, \int_0^s \dcal_{\half}(t)dt]  + \dcal_0. \end{equation}
The diagonal part is the part  which is a function only of $|z|^2$. The off-diagonal terms have not previously
been studied because they can be eliminated from the normal form in the non-degenerate case. Their vanishing
in the MDL Zoll case is additional information on the metric.

We now prove the following Lemma and also evaluate the coefficients:

\begin{lem} \label{DCALCOMMLEM}

The diagonal part of \eqref{DCALCOMM}  term has the form
$$ \begin{array}{l} \mbox{Diagonal Part} \left( \frac{1}{2}\int_0^L\int_0^s[\tilde{D}_{\half}(s,z,\bar{z}),
\tilde{D}_{\half}(t,z,\bar{z}]dsdt \right) = A_4 |z|^4 + A_0, \end{array} $$
with 
$$\begin{array}{l} A_4 = [\int_0^L\int_0^s  C_{1;3030}
d_{\half;30}(s)d_{\half; 03}(t) + C_{1;2112}d_{\half;21}(s)d_{\half; 12}(t)
\\ \\+C_{1;1221}
 d_{\half;12}(s)d_{\half;21}(t) +
C_{1;0330}d_{\half;30}(s)d_{\half; 03}(t) dsdt]  \end{array} $$
and with 
$$\begin{array}{l} 
A_0 =  [\int_0^L\int_0^s [C_{3;3030} d_{\half; 30}(s)d_{\half; 03}(t) \\ \\  +
C_{3;2112}d_{\half; 21}(s)d_{\half; 12}(t) +C_{3;1221}
d_{\half;12}(s)d_{\half;21}(t) + C_{3;0330}d_{03}(s)d_{\half;30}(t) ]dsdt].  \end{array}$$ The coefficients are universal. 
\end{lem}

Making the linear symplectic substitutions above we first get
\begin{equation} \label{DCALHALFa} \dcal_{\half}(s,z,\bar{z}) = C \tau_{\nu}(s)
 ( [ \bar {Y}\cdot z + Y \cdot \bar{z}])^{3}. \end{equation}
We note that $\dcal_{\half}(s)$  is a homogeneous polynomial of degree 3. Obviously,

\begin{lem} \label{dhalf} For $m + n = 3, $ the coefficient of $z^m \bar{z}^n$ in  $\dcal_{\half}(s)$ is
$$d_{\half; mn}(s) =  C_{mn;3} \tau_{\nu} [\bar
{Y}^m\cdot Y^n](s).$$
\end{lem}

Let $[F, G] = F \# G - G \# F$. Then the commutator in  \eqref{DCALCOMM}  equals

\begin{equation} \label{DCALHALF} \begin{array}{l} 
 \sum_{m + n = 3, p + q = 3}  C_{m,n,p,q}  
\left(\int_0^L \int_0^s d_{\half; mn}(s) d_{\half; pq}(t)  ds dt \right) [z^m \bar{z}^n, 
z^p \bar{z}^q]\\ \\ = \sum_{m + n = 3, p + q = 3}  C'_{m,n,p,q}\{[\int_0^L \tau_{\nu}(s)
 [\bar
{Y}^m\cdot Y^n](s) \int_0^s \tau_{\nu}(t)
 [\bar
{Y}^p\cdot Y^q](t)  dt \}  [z^m \bar{z}^n, 
z^p \bar{z}^q]. \end{array} \end{equation}
The  commutator $  [z^m \bar{z}^n, 
z^p \bar{z}^q]$ is described in \S \ref{C} and is a sum of a Poisson bracket and a third
transvectant.



\subsubsection{Diagonal terms  of the commutator}

 We note that
$\{z^m \bar{z}^n, z^p \bar{z}^q \} = C z^{m + p -1} \bar{z}^{n + q -1}$ and since $\dcal_{\half}$ is homogeneous
of degree 3, to obtain a term of type $|z|^{2k}$ ($k = 0, 1, 2$) we need $m + p  = k +1 = n + q , m + n = 3 = p + q.$ 
The only non-zero Poisson bracket  occurs when  $k = 2$ and then $m + p = 3 = n + q$, so $p = n, q = m$. This produces terms of the form $[z^m \bar{z}^n, z^n \bar{z}^m]$ times
$$ \{[\int_0^L \tau_{\nu}(s)
 [\bar
{Y}^m\cdot Y^n](s) \int_0^s \tau_{\nu}(t)
 [\bar
{Y}^n\cdot Y^m](t)  dt \}.  $$
If we interchange $m \to n, n \to m$, the diagonal terms of the commutator $[z^m \bar{z}^n, z^n \bar{z}^m]$ changes
sign and the above integral changes to its complex conjugate. Hence, the diagonal part of the commutator
produces a sum of terms, 
$$\sum_{m + n = 3} C_{m, n} \Im  \{[\int_0^L \tau_{\nu}(s)
 [\bar
{Y}^m\cdot Y^n](s) \int_0^s \tau_{\nu}(t)
 [\bar
{Y}^n\cdot Y^m](t)  dt \}.  $$

But additionally,  the symbol of the commutator involves the third transvectant $P_3$
of \S \ref{C}, and the $P_3$ of any two monomials of degree 3 is a constant in $z$.


We 
observe that there is no term of order $|z|^2$ since the Poisson bracket of two 
homogeneous polynomials of degree 3 has degree 4 and the $P_3$-transvectant has
degree zero. This proves Lemma \ref{DCALCOMMLEM} and also: 



\begin{lem} \label{COMM4} On a Zoll surface, the coefficient of  $|z|^4$  or $1$   in the commutator term 
\eqref{COMMPART} is a sum of universal constants times
$$\begin{array}{l} \frac{1}{L}\{\int_0^L \tau_{\nu}(s)\bar {Y}^mY^n(s)[\int_0^s\tau_{\nu}(t)\bar
{Y}^nY^m](t)dt] ds - \int_o^L \tau_{\nu}(s)\bar {Y}^nY^m(s)[\int_0^s\tau_{\nu}(t)\bar
{Y}^mY^n(t)dt] ds\} \\ \\
= 2 \frac{1}{L} \Im  \int_0^L \tau_{\nu}(s)\bar {Y}^mY^n(s)[\int_0^s\tau_{\nu}(t)\bar
{Y}^nY^m](t)dt] ds. \end{array}$$

\end{lem}


\begin{rem} \label{OD1} There are many `off-diagonal' terms in the commutator
corresponding to $[z^m \bar{z}^n, z^p \bar{z}^q]$ which are not a power of $|z|^2$. These
obviously produce terms of the form
$$ \{[\int_0^L \tau_{\nu}(s)
 [\bar
{Y}^m\cdot Y^n](s) \int_0^s \tau_{\nu}(t)
 [\bar
{Y}^q\cdot Y^p](t)  dt \}.  $$
On a Zoll surface the complete sum of the off-diagonal terms must vanish. 
\end{rem}

\subsubsection{Diagonal terms of   $\dcal_0|_0$}

To  complete the calculation of \eqref{DCALHALFINT}, we need to find the symbol of the $D_s$-weight 0 part   $\dcal_0|_0$  of the remaining term of $\dcal_0^{\half}$  in \eqref{DCALCOMM}.  We 
make the same linear substitution and eliminate any $D_s$ appearing all the way to the right.
 We also invert the relation
$$\mu(\acal^*)^{-1} D_s \mu(\acal^*) = D_s - \half( \partial_x^2 +  \tau x^2)$$
to get
$$\mu(\acal^*) D_s \mu(\acal^*)^{-1} = D_s - \half\mu(\acal^*)(\partial_x^2 +  \tau x^2)
\mu(\acal^*)^{-1}$$
and transform the complete symbol of quadratic term by the symplectic substitution.  The result is
that
\begin{lem} \label{dcal0}  $\dcal_0|_0 (s,z,\bar{z})$ equals
$$ \begin{array}{ll} (i) & C_1 \tau_{\nu \nu}  [\bar {Y} z + Y  \bar{z}]^4  + C_3 \tau_s [\bar {Y} z + Y  \bar{z}]^2 \\ \\
(ii) & +C_2 \tau [\bar {Y} z + Y  \bar{z}]^2 \#
( [\bar{\dot{Y}} z + \dot{Y}\bar{z}]^2 + \tau [\bar {Y} z + Y \bar{z}]^2))
 \\ \\
(iii) &  +\{ [\bar{\dot{Y}} z + \dot{Y}\bar{z}]^2 - \tau  [\bar {Y}\cdot z + Y \bar{z}]^2\}
\#\{ [\bar{\dot{Y}} z + \dot{Y}\bar{z}]^2 - \tau [\bar {Y} z + Y  \bar{z}]^2\} \\ \\
(iv) &   -2 \partial_s( [\bar{\dot{Y}} z + \dot{Y}\bar{z}]^2 -\tau [\bar {Y} z + Y
\bar{z}]^2) \\ \\
(v) & + C_4 \tau (\bar {Y} z + Y  \bar{z})\#(\bar{\dot{Y}} z + \dot{Y}\bar{z}) +C_5 \tau. \end{array}$$
\end{lem}

Our concern is with the diagonal part of the complete symbol, that is, with the terms involving
$|z|^4, |z|^2, |z|^0$, and more precisely with their integrals over $\gamma$. The
$\#$ product  produces a finite sum  terms of decreasing degree corresponding to the
higher transvectants. 

The
 diagonal part of the first term of  (i)   is homogeneous  of degree $|z|^4$ and its average over $\gamma$ equals
$$ (Const.) |z|^4 \cdot \frac{1}{L} \int_0^L \tau_{\nu \nu} |Y|^4 ds. $$ It is the "$d_j$'' 
term in the formula of Theorem \ref{GBO} and is clearly only a coefficient of $|z|^4$; this
explains why the term does not appear as a coefficient of $|z|^0$.  The integral of
the diagonal part of  the second
term of (i) vanishes, 
$$\;\;\;\;\;\;\;\;\int_0^L \tau_s |Y|^2=0,$$
as can be seen from the  Jacobi equation, which implies:
$$ [\bar{Y} (Y')'' + \tau_s|Y|^2 + \tau Y'\bar{Y}]=0;$$
integrating over $\gamma$ and integrating the first term by parts twice kills the outer terms and
hence the inner one.

Terms (ii) and (iii) are similar and possibly combine. Term (iii) is a $\#$-square, as is one of
the two terms of (ii). 

The diagonal part of  the $\#$ product in (ii)  contributes only  $P_0$ and $P_2$ terms, of degrees $|z|^4$ and $|z|^0$
respectively. The $|z|^4$ coefficient is that of 
$$\tau^2 [\bar {Y} z + Y  \bar{z}]^4 + \tau [\bar {Y} z + Y  \bar{z}]^2
 [\bar{\dot{Y}} z + \dot{Y}\bar{z}]^2 = (\tau^2 |Y|^4  + \tau |Y \dot{Y}|^2)|z|^4 + \cdots, $$
where the $\cdots$ terms are homogeneous of degree 4 but do not contain a term of type $|z|^4$. 
The integral over $\gamma$ of the possible  $P_1$-term of type $|z|^2$  vanishes: it is a multiple of the Poisson bracket
$$P_1 ([\bar {Y} z + Y  \bar{z}]^2, \tau [\bar {Y} z + Y \bar{z}]^2))$$
which simplifies to a term of the form
$$\tau [\bar{Y}^2\dot{Y}^2 - Y^2\dot{\bar {Y}}^2] = \tau (\bar{Y}\dot{Y} - Y \bar{\dot{Y}})(\bar{Y}\dot{Y} +
Y\bar{\dot{Y}}) = C \tau \frac{d}{ds} |Y|^2$$
by the symplectic normalization of the Jacobi eigenfield.   As mentioned above, its
integral vanishes. The $P_0$ term is obtained by applying the square of the bi-differential operator
$$\sum_j (\partial_{z_j} \partial_{\bar{w_j}} - \partial_{\bar{z}_j} \partial_{w_j}) f(z)g(w) $$
to the expression with $f$ on the left and $g$ on the right and setting $z = w$.

  The term
(iii)  is a homogeneous $\#$-square, hence its diagonal part 
contains  only a product $P_0$-term of degree $|z|^4$ and a
$P_2$-term of degree $0$, namely (for j=0,2) the diagonal part of
$$\begin{array}{l} P_j [z^2\dot{\bar{Y}}^2 + 2|z|^2|\dot{Y}|^2 +\bar{z}^2\dot{Y}^2 -  \tau(z^2\bar{Y}^2
 + 2 |z|^2 |Y|^2 + \bar{z}^2Y^2), \\ \\z^2\dot{\bar{Y}}^2 + 2|z|^2|\dot{Y}|^2 +\bar{z}^2\dot{Y}^2 - \tau(z^2\bar{Y}^2
 + 2 |z|^2 |Y|^2 + \bar{z}^2Y^2) ] \end{array}$$
whose average over $\gamma$ has the form for $j = 2, 0$
$$ \int_0^L [a_j |\dot{Y}|^4
+ b_{1 j} \tau 2 \Re (\dot{\bar{Y}}^2 Y^2) + 2 b_{2j} \tau |\dot{Y}Y|^2
+c_j  \tau^2 |Y|^4]ds $$
where
$$a_j= c_j  =  2 P_j(z^2, \bar{z}^2) + 2 P_j(|z|^2, |z|^2) , \;\; b_{1j}  =  - 2P_j(z^2, \bar{z}^2) = b_{2j}
$$
Note here than $P_j$ is symmetric if $j$ is even and that $P_2(z^2, \bar{z}^2) = - 
2 P_2(|z|^2, |z|^2)$. Also
$$ 2 \Re (\dot{\bar{Y}}^2 Y^2) + 2   |\dot{Y}Y|^2 = (\dot{Y} \bar{Y} + \overline{\dot{Y}} Y)^2. $$
In Lemma \ref{4Id} we show that
$$\int_{\gamma} |\dot{y}|^4 ds  
 = 2 \int \tau |Y \dot{Y}|^2   ds +  \Re  \int_{\gamma} \tau (\dot{Y} \bar{Y})^2 ds. $$ 
 The  diagonal
    term  of (iv) has vanishing integral  since it   is  a total derivative. 
Finally, we  the first term of (v) obviously has no diagonal
part while obviously the second term contributes the zeroth order term 
$$C  \int_{\gamma} \tau ds. $$

Adding  the terms of  Lemma \ref{COMM4} and Lemma \ref{dcal0} 
completes the analysis of the Birkhoff normal form coefficient $f_0(I)$ in the Zoll case.
It must vanish in the maximally degenerate case.

\subsection{Off-diagonal terms}

In the MDL Zoll case, existence of the normal form implies that the off-diagonal terms must vanish and
triviality of the normal form implies that the diagonal terms vanish. Above we emphasized the diagonal terms.
We now briefly consider the off-diagonal ones. 

From the commutator terms \eqref{DCALHALF} we obtain off-diagonal terms 

\begin{equation} \label{DCALHALFb} \begin{array}{l}   \sum_{m + n = 3, p + q = 3}  C'_{m,n,p,q}\{[\int_0^L \tau_{\nu}(s)
 [\bar
{Y}^m\cdot Y^n](s) \int_0^s \tau_{\nu}(t)
 [\bar
{Y}^p\cdot Y^q](t)  dt \}  [z^m \bar{z}^n, 
z^p \bar{z}^q], \end{array} \end{equation}
where $[z^m \bar{z}^n, 
z^p \bar{z}^q] = \{z^m \bar{z}^n, 
z^p \bar{z}^q\} + C_3 P_3 (z^m \bar{z}^n, 
z^p \bar{z}^q). $ Thus we obtain the same geometric invariants as in the diagonal case, but with different coefficients determined
by transvectants of monomials. The Poisson bracket gives all possible monomials of degree 4 and $P_3$ gives the two
mononomials of degree 1. 

Lemma \ref{dcal0}  expresses   $\dcal_0|_0 (s,z,\bar{z})$ as a sum of compositions involving $P_4, P_3, P_2, P_1, P_0$
where $P_0 $ is simply multiplication. We then obtain a sum of monomials of degrees $\leq 4$. Adding to the commutator
terms gives coefficients of monomials $z^a \bar{z}^b$ with $a + b \leq 4$ and the coefficient must vanish for every
monomial.  We worked out the monomials $1, |z|^2, |z|^4$ in detail because they are the ones that arise in
the Birkhoff normal form construction in the non-degenerate case. But the vanishing of the  off-diagonal monomial coefficients
is just as informative as the vanishing of the diagonal ones. Since the transvectant coefficients differ, it is likely
that by taking linear combinations of all of the invariants defined as coefficients of monomials, we can simplify
the condition of Theorem \ref{GBO}.

\section{\label{J} Final Remarks}

There are many open problems regarding Zoll Laplacians, some of which might
not be  difficult to answer. As mentioned before, they are tests of the known techniques
in inverse spectral theory.  We close with some observations and speculations
as well as with some further identities which simplify Theorem \ref{GBO}.

\subsection{Analysis of $p_1(q)$}

We observe that  each the integrals over closed geodesics in the expression for 
$p_1(q)$ in Theorem \ref{GBO}  defines a function on $S^*_g S^2$. To see this, we observe
that 
 $(x, \xi)$ determines the closed geodesic $\gamma_{x, \xi}(t) = \pi G^t(x, \xi)$ where $G^t$ is the geodesic
flow. The basis of  Jacobi fields $y_1, y_2$ with Wronskian matrix \eqref{as} equal to the identity
at $t = 0$ is then determined uniquely by $(x, \xi)$.
The various terms above are integrals over $\gamma_{x, \xi}$ and thus each term is a function of $(x, \xi)$. 
Also $\nu = J \dot{\gamma}$ is defined by $(x, \xi)$ and the unique complex structure $J$ of $S^2$. 

Similarly, each term of the  integrand is a function of $(x, \xi, t) \in S^*_g S^2 \times S^1$. 
The entire integrand for the coefficient $c_{2j}$ of $p_1$  defines a function $F_{ j}(x, \xi, t) $ on $S^*_g S^2 \times S^1$ depending on the index $j$. In the maximally degenerate case,
or more generally when $p_1 = 0$,   $F_{2j} = \frac{\partial}{\partial t} f_{2 j}$ for some smooth
function $f_{2j}$. Equivalently, if $\Xi$ generates the geodesic flow, then $F_{2j} = \Xi f_{2j}$.  If we expand
$F$ as a Fourier series in the $t$ variable, we have
$$F(x, \xi, t) = \sum_{n \in \Z. n \not= 0} \hat{F}(x, \xi, n) e^{i n t}, \;\; f(x, \xi, t) = \sum_{n \in \Z. n \not= 0}\frac{\hat{F}(x, \xi, n)}{in}  e^{i n t}. $$
 But $F$ is constructed in a universal way from the Zoll metric.

This condition is satisfied by the standard metric.
An obvious question is whether $\sigma_{A_0} = 0$ for any non-standard Zoll metric $g$. In fact, it is
not even obvious whether there may exist a deformation of the standard metric $g_0$
through Zoll metrics for which $\sigma_{A_0} (g_t) = 0$ for $t \in [0, \epsilon]$ for some $\epsilon > 0$. It is hoped that Theorem \ref{GBO} could provide
some information on this question. In fact it is not even clear that there exist any Zoll surfaces
for which $\sigma_{A_0} (g) \not= 0$ although it was found in \cite{Z1} that no Zoll surfaces of revolution
have this property.

In the case of potential perturbations $\Delta_0 + V$ on the standard $S^2$ the analogous symbol is
$\int_{\gamma}  V ds$ and it is not hard to show that only $V = 0$ is maximally degenerate.
The proof uses the explict formula for the normal form of the operator. This suggests that we must rely on the expression
in Theorem \ref{GBO} in the metric case. 

Of course, in the MDL case all of the terms of the Birkhoff normal form expansion vanish,
and the Zoll Laplacian has the form $\rcal(\rcal + 1) + \scal$ where $\scal$ is smoothing
and is a function of $\rcal$. In \cite{Z1}  the trace of $\scal$ is determined and in principal
one could determine $Tr \scal^n$ for any $n$ using the heat kernel expansion
for $$\begin{array}{lll} \exp - t (\rcal(\rcal + 1) + \scal) & = & \exp (- t \scal) \exp (- t \rcal(\rcal + 1))
\\&&\\
& =  & \sum_{n = 0}^{\infty} \frac{(-t)^n}{n!} Tr \scal^n \exp (- t \Delta_0)\\ &&\\ & \sim & 
Tr e^{- t \Delta_0} + 
\sum_{n = 1}^{\infty} \frac{(-t)^n}{n!} \sum_{m = 0}^{\infty} \frac{(-t)^m}{m!} Tr \scal^n \Delta_0^m,
\;\; (t \to 0), 
\end{array}$$
where $\Delta_0 = \rcal(\rcal + 1)$ has the same
 eigenvalues and multiplicities as the standard Laplacian. One  might try to determine 
$\scal(k)$ completely from these traces using the heat kernel expansion,  or prove that no such $\scal(k)$ can exist. However, there are many terms producing a given order of $t^M$.

One might also think in terms of the infinite dimensional moduli space $\zcal$ of Zoll
metrics on $S^2$ \cite{G3}. Its tangent space at a Zoll metric $g$ consists of the kernel of the
geodesic Radon transform of $g$ on $C^{\infty}(S^2)$. 
In the case of the standard metric,
the kernel consists of the odd functions.
The function $F(x, \xi, t; g)$ is constructed in a universal way from the 4-jet  $g, Dg, D^2 g, D^3 g, 
D^4 g$  of
the metric on the moduli space of Zoll metrics. Hence the terms of $p_1(g)$ define `natural'
maps from $\zcal \to C^{\infty}(S^*_g M )$ with equivariance under $G^t$, so $p_1(g)$
is a natural function on $G(M, g)$. Of course, the symbol of $A_{-1}$ is of this kind. There might
exist a geometric  functional on $\zcal$ which is critical exactly at MDL metrics.  But since
it is a spectral condition, the functional would most likely be a spectral one. 


\subsection{Balancing issues}

In \cite{Z3} this question was approached via the spectral projections kernels $\Pi_k(x, y)$ for the 
eigenvalue clusters rather than through Gaussian beams and normal forms. In the maximally
degenerate case, the eigenmap embeddings $y \to \Pi_k(\cdot, y): S^2 \to  \hcal_k$ are almost isometric
up to order $k^{-\infty}$. Here $\hcal_k$ is the Hilbert space of eigenfunctions of
$\Delta_g$ for the kth cluster (in the MDL case); see \S \ref{MDL} for background. The metric $g$ would have to be the standard one if the maps
were exactly isometric. In the Gaussian beam setting we instead have the maps
$$H_k: G(S^2, g) \simeq S^2 \to \hcal_k, \;\;\; H_k(\gamma) = \phi_k^{\gamma},  $$
although to be precise $\phi_k^{\gamma}$ is defined only up to a constant (depending
on a choice of basepoint for $\gamma$), and it would be better to define the map
to the projective space of $\hcal_k$. Again, the map $H_k$ is almost isometric in
the MDL case, as follows by combining the formula for $\Pi_k$ in \S \ref{MDL} with
the results of \cite{Z3}.

The article
preceded \cite{Z3} predated  the modern era of balanced K\"ahler metrics and Bergman metrics but
is somewhat  related in spirit.  There is a well-known analogy between Zoll surfaces and positive
line bundles over K\"ahler surfaces in which the unit tangent bundle of the Zoll surface
is parallel to the Hermitian metric unit bundle. In both cases the unit bundle is an $S^1$-bundle. Indeed, it is more than an analogy since the unit tangent
bundle is a circle bundle over the space $G(S^2, g)$ of geodesics of $g$ and is the unit
bundle for a  Hermitian metric of positive $(1,1)$ curvature on the  holomorphic line $T \CP^1$ over $G(S^2, g)$. It is not clear
what if any implications maximal degeneracy of $\Delta_g$  has for the Szeg\"o kernel of this line bundle or for its  $\dbar$-Laplacian. For instance, the spectral projections kernel $\Pi_k$
of a MDL is almost constant on the diagonal (modulo a term of order $k^{- \infty}$). Is
the same true for the associated Szeg\"o kernel of $H^0(\CP^1, (T\CP^1)^{\otimes k})$?

\subsection{\label{RELATIONS}Relations among terms}

We note some relations between the integrals in Theorem \ref{GBO}.  They allow
for some simplification of the expression for $p_1(q)$ and suggest that other simplifications
might exist. They also help in comparing Theorem \ref{GBO} with the formulae in
\cite{Z1,Z3} but for the sake of brevity we do not make a detailed comparison.

\begin{lem}  For any solution of the Jacobi equation of any Zoll surface,
$$ \int \tau y^2 \dot{y}^2   ds=  \frac{1}{3} \int (\dot{y})^4 ds. $$
\end{lem}
To prove this we multiply the Jacobi equation by $\dot{y}^2 y$ and use that
$$\int \dot{y}^2 y \ddot{y} ds  = - \int \dot{y} \frac{d}{ds} (y \dot{y}^2) ds = -  \int (2 \dot{y}^2 y \ddot{y} - \dot{y}^4) d. $$
Jacobi's equation implies that 
$$- \int \dot{y}^2 y \ddot{y}  ds=  \int \tau y^2 \dot{y}^2   ds.  $$
Combining the formulae concludes the proof.

A variant that applies directly to Theorem \ref{GBO} is the following.
\begin{lem} \label{4Id}  For an Zoll surface and close geodesic
$$\Im  \int_{\gamma} \tau (\dot{y} \bar{y})^2 ds = 0, \;\; \int_{\gamma} |\dot{y}|^4 ds  - 
2 \int \tau |y \dot{y}|^2   ds = \Re  \int_{\gamma} \tau (\dot{y} \bar{y})^2 ds. $$ 
This allows us to remove the third term (the $b_2$ term) and with universal changes in the coefficients
of the first two terms.
\end{lem}

\begin{proof}


Consider the term,
$$
\int \tau |y \dot{y}|^2   ds = \int \tau y \bar{y} \dot{y} \overline{\dot{y}} ds. $$
Multiply the Jacobi equation for $y$ by $\bar{y} \dot{y} \overline{\dot{y}}$. 
Jacobi's equation implies that 
$$
\int \tau |y \dot{y}|^2   ds = -  \int \ddot{y} \bar{y} \dot{y} \overline{\dot{y}} ds.$$
We also have 
$$\begin{array}{lll} -  \int \ddot{y} \bar{y} \dot{y} \overline{\dot{y}} ds   & = &   \int \dot{y} \frac{d}{ds} (\bar{y} \dot{y} \overline{\dot{y}}) ds 
\\ &&\\ && = \int \dot{y} (\frac{d}{ds} (\bar{y} )  \dot{y} \overline{\dot{y}}) ds 
+ \int \dot{y} (\bar{y} \frac{d}{ds} \dot{y} )\overline{\dot{y}}) ds + \int \dot{y} (\bar{y} \dot{y} \frac{d}{ds} \overline{\dot{y}}) ds 
\\&&\\
& = &   \int \dot{y}^2\overline{\dot{y}}^2 ds 
- \int \dot{y} (\bar{y} \tau y )\overline{\dot{y}}) ds - \int \dot{y} (\bar{y} \dot{y} \tau  \overline{y}) ds  . \end{array}$$
Hence
$$
2 \int \tau |y \dot{y}|^2   ds =  \int |\dot{y}|^4 ds 
- \int \tau (\dot{y} \bar{y})^2 ds. $$

All the terms are real except the last one, so its imaginary part is zero. 

\end{proof}

We may also simplify the ``commutator term'' of Theorem \ref{GBO}. A real orthogonal  Jacobi
field $Y = y \nu$ along $\gamma$ defines a variation $\gamma_r(t)$ of $\gamma(t)$
with variational vector field $Y$. We then have a 1-parameter family of Jacobi fields
$Y_r(t)$ on $\gamma_r(t)$ and can differentiate with respect to $r$. We denote the
derivative of $y_r$ by $y_{\nu}$ (possibly not good notation). 
For any Jacobi field  we get the equation  $ y_{\nu}'' + \tau_{\nu} y^2 + \tau y_{\nu} = 0 $
for this varation. Multiply by a section Jacobi field  $y_2$ to get 
 $y_{\nu}'' y_2 + \tau_{\nu} y^2 y_2 + \tau y_{\nu} y_2= 0 $ Therefore
$$\tau_{\nu} y^2 y_2 = - (y_{\nu}'' y_2 + \tau y_{\nu} y_2). $$
It follows that 
$$ \begin{array}{lll} - \int_0^s\tau_{\nu}(t)\bar {Y}^nY^m(t)dt & = &  \int_0^s (y_{\nu}'' y_2 + \tau y_{\nu} y_2) )dt \\&&\\
&&
=  y_{\nu}' y_2 |_0^s - \int_0^s (y_{\nu}' y_2' + \tau y_{\nu} y_2)\\&&\\&&
=  y_{\nu}' y_2 |_0^s -  (y_{\nu} y_2') |_0^s = Y_{\nu}'(s) \overline{Y}(s) - Y_{\nu}(s) \overline{Y}' (s) 
- Y_{\nu}'(0) + i  Y_{\nu}(0).  \end{array}$$
Since $Y = y_1 + i y_2, \overline{Y} = y_1 - i y_2$ if $y_2 = \overline{Y}$ we have $y_2(0) = 1, y_2'(0) = - i$. 
Since $Y(0) = 1, Y'(0) = i$ for all $\gamma$, the variations of these quantities vanish and the last two terms
are zero.   Therefore,
$$\begin{array}{l} \Im \{\int_0^{2 \pi} \tau_{\nu}(s)\bar
{Y}^mY^n(s)[\int_0^s\tau_{\nu}(t)\bar {Y}^nY^m](t)dt] ds\} \\ \\   =
\Im \{\int_0^{2 \pi} \tau_{\nu}(s)\bar
{Y}^mY^n(s)[ Y_{\nu}'(s) \overline{Y}(s) - Y_{\nu}(s) \overline{Y}' (s) ] ds\} \\ \\ 
   =
 \int_0^{2 \pi} \tau_{\nu}(s) \Im \{\bar
{Y}^mY^n(s)[ Y_{\nu}'(s) \overline{Y}(s) - Y_{\nu}(s) \overline{Y}' (s) ]\} ds 
 .\end{array} $$
For instance, if $m = 2, n = 1$ the ``commutator term'' is equal to
$$ \int_0^{2 \pi} \tau_{\nu}(s) \Im \{\bar
[ Y(s) Y_{\nu}'(s) \overline{Y}^3(s) - Y(s) Y_{\nu}(s) \overline{Y(s)}^2  \overline{Y}' (s) ]\} ds  $$

Also, the Wronskian
condition
$$Y(s) \overline{Y}'(s) - Y'(s)  \overline{Y}(s) = i $$
holds for all $\gamma$ and therefore its variation vanishes. Thus,
$$Y_{\nu} (s) \overline{Y}'(s) - Y_{\nu}'(s)  \overline{Y}(s)  + Y(s) \overline{Y}'_{\nu} (s) - Y'(s)  \overline{Y}_{\nu}(s) 
= 2 \Im (Y_{\nu} (s) \overline{Y}'(s) - Y_{\nu}'(s)  \overline{Y}(s) )   = 0. $$ 
Thus,
$$ \int_0^{2 \pi} \tau_{\nu}(s) \Im \{\bar
{Y}^mY^n(s)[ Y_{\nu}'(s) \overline{Y}(s) - Y_{\nu}(s) \overline{Y}' (s) ]\} ds  =
 \int_0^{2 \pi} \tau_{\nu}(s) \Im \{\bar
{Y}^mY^n(s)\} [ Y_{\nu}'(s) \overline{Y}(s) - Y_{\nu}(s) \overline{Y}' (s) ] ds . $$

\end{document}